\documentclass[dvipsnames,hidelinks,onefignum,onetabnum]{siamart220329}

\usepackage{amsfonts,amssymb}
\usepackage{graphicx}
\usepackage{epstopdf}
\usepackage{mathtools}
\usepackage{accents}
\usepackage{tikz}
\usepackage{algorithmicx}
\usepackage{algpseudocode}
\usepackage{enumitem}
\algnewcommand\algorithmicinput{\textbf{Input:}}
\algnewcommand\Input{\item[\algorithmicinput]}
\algnewcommand\algorithmicoutput{\textbf{Output:}}
\algnewcommand\Output{\item[\algorithmicoutput]}
\ifpdf
  \DeclareGraphicsExtensions{.eps,.pdf,.png,.jpg}
\else
  \DeclareGraphicsExtensions{.eps}
\fi

\newcommand{\N}{\mathbb{N}}
\newcommand{\R}{\mathbb{R}}

\DeclareMathOperator{\im}{im}
\newcommand{\dist}{d}
\newcommand{\ball}{B}
\newcommand{\norm}[1]{\Vert #1 \Vert}
\newcommand{\ip}[2]{\left\langle #1, #2 \right\rangle}
\DeclareMathOperator*{\argmin}{argmin}
\DeclareMathOperator*{\argmax}{argmax}
\DeclareMathOperator*{\lip}{Lip}

\newcommand{\proj}[2]{P_{#1}(#2)}

\newcommand{\gencone}[4]{{#1}_{#2}^{#4}(#3)} %general cone
\newcommand{\tancone}[2]{\gencone{T}{#1}{#2}{}} %tangent cone
\newcommand{\norcone}[2]{\gencone{N}{#1}{#2}{}} %normal cone
\newcommand{\regnorcone}[2]{\gencone{\widehat{N}}{#1}{#2}{}} %regular normal cone
\newcommand{\proxnorcone}[2]{\,\hspace{1pt}\gencone{\widehat{\!\hspace{-1pt}\vphantom{\rule{1pt}{8.5pt}}\smash{\widehat{N}}\hspace{-1pt}}}{#1}{#2}{}} %proximal normal cone

\newcommand{\oshort}[1]{\mkern 0.9mu\overline{\mkern-0.9mu#1\mkern-0.9mu}\mkern 0.9mu}
\newcommand{\ushort}[1]{\mkern 0.9mu\underline{\mkern-0.9mu#1\mkern-0.9mu}\mkern 0.9mu}

\newcommand{\sparse}[2]{\R_{\le #2}^{#1}}
\DeclareMathOperator{\supp}{supp}

\newcommand{\tp}{\top}
\DeclareMathOperator{\rank}{rank}

\newcommand{\pgd}{\mathrm{PGD}}
\newcommand{\ppgd}{\mathrm{P}^2\mathrm{GD}}

\newsiamremark{remark}{Remark}
\newsiamremark{hypothesis}{Hypothesis}
\crefname{hypothesis}{Hypothesis}{Hypotheses}
\newsiamthm{example}{Example}

\headers{PGD accumulates at B-stationary points}{G. Olikier and I. Waldspurger}

\title{Projected gradient descent accumulates at Bouligand stationary points\thanks{Submitted to the editors September 11, 2024.
\funding{This work was supported by ERC grant 786854 G-Statistics from the European Research Council under the European Union's Horizon 2020 research and innovation program and by the French government through the 3IA Côte d'Azur Investments ANR-23-IACL-0001 and the PRAIRIE 3IA Institute ANR-19-P3IA-0001, managed by the National Research Agency.}}}

\author{Guillaume Olikier\thanks{Université Côte d'Azur and Inria, Epione Project Team, 2004 route des Lucioles - BP 93, 06902 Sophia Antipolis Cedex, France
  (\email{guillaume.olikier@inria.fr}).}
\and Irène Waldspurger\thanks{CNRS, Universit\'e Paris Dauphine, équipe-projet Mokaplan (Inria), place du Maréchal de Lattre de Tassigny, 75016 Paris, France
  (\email{waldspurger@ceremade.dauphine.fr}).}
}

\usepackage{amsopn}
\DeclareMathOperator{\diag}{diag}

\ifpdf
\hypersetup{
  pdftitle={PGD accumulates at B-stationary points},
  pdfauthor={G. Olikier and I. Waldspurger}
}
\fi

\begin{document}

\maketitle

% REQUIRED
\begin{abstract}
This paper considers the projected gradient descent (PGD) algorithm for the problem of minimizing a continuously differentiable function on a nonempty closed subset of a Euclidean vector space. Without further assumptions, this problem is intractable and algorithms are only expected to find a stationary point. PGD generates a sequence in the set whose accumulation points are known to be Mordukhovich stationary. In this paper, these accumulation points are proven to be Bouligand stationary, and even proximally stationary if the gradient is locally Lipschitz continuous. These are the strongest stationarity properties that can be expected for the considered problem.
\end{abstract}

% REQUIRED

% REQUIRED

\section{Introduction}
\label{sec:Introduction}
Let $\mathcal{E}$ be a Euclidean vector space, $C$ a nonempty closed subset of $\mathcal{E}$, and $f : \mathcal{E} \to \R$ a function satisfying at least the first of the following two properties:
\begin{enumerate}[label={(H\arabic*)}]
\item $f$ is differentiable on $C$, i.e., for every $x \in C$, there exists a (unique) vector in $\mathcal{E}$, denoted by $\nabla f(x)$, such that
\begin{equation*}
\lim_{\substack{y \to x \\ y \in \mathcal{E} \setminus \{x\}}} \frac{f(y)-f(x)-\ip{\nabla f(x)}{y-x}}{\norm{y-x}} = 0,
\end{equation*}
and $\nabla f : C \to \mathcal{E}$ is continuous;
\label{it:continuously_diff}
\item $f$ is differentiable on $\mathcal{E}$ and $\nabla f : \mathcal{E} \to \mathcal{E}$ is locally Lipschitz continuous.
\label{it:Lipschitz_diff}
\end{enumerate}
This paper considers the problem
\begin{equation}
\label{eq:MinDiffFunctionClosedSet}
\min_{x \in C} f(x)
\end{equation}
of minimizing $f$ on~$C$.
In general, without further assumptions on~$C$ or $f$, finding an exact or approximate global minimizer of problem~\eqref{eq:MinDiffFunctionClosedSet} is intractable. Even finding an approximate local minimizer is not always feasible in polynomial time (unless $\text{P} = \text{NP}$) \cite{AhmadiZhang}. Therefore, algorithms are only expected to return a point satisfying a condition called \emph{stationarity}, which is a tractable surrogate for local optimality.

A point $x \in C$ is said to be stationary for~\eqref{eq:MinDiffFunctionClosedSet} if $-\nabla f(x)$ is normal to $C$ at~$x$. Several definitions of normality exist. Each defines a different notion of stationarity, which is a surrogate for local optimality in the sense that, possibly under mild regularity assumptions on~$f$, every local minimizer of $f|_C$ is stationary for~\eqref{eq:MinDiffFunctionClosedSet}. In particular, each of the three notions of normality in \cite[Definition~6.3 and Example~6.16]{RockafellarWets}, namely normality in the general sense, in the regular sense, and in the proximal sense, yields an important definition of stationarity. The sets of general, regular, and proximal normals to $C$ at $x \in C$ are respectively denoted by $\norcone{C}{x}$, $\regnorcone{C}{x}$, and $\proxnorcone{C}{x}$. These sets are reviewed in section~\ref{subsec:NormalityStationarity}. Importantly, they are nested as follows: for every $x \in C$,
\begin{equation}
\label{eq:NestedNormalCones}
\proxnorcone{C}{x} \subseteq \regnorcone{C}{x} \subseteq \norcone{C}{x},
\end{equation}
and $C$ is said to be \emph{Clarke regular} at $x$ if the second inclusion is an equality.
The definitions of stationarity based on these sets are given in Definition~\ref{def:M/B/P-Stationarity}, and the terminology is discussed in section~\ref{sec:StationarityLiterature}.

\begin{definition}
\label{def:M/B/P-Stationarity}
For problem~\eqref{eq:MinDiffFunctionClosedSet}, a point $x \in C$ is said to be:
\begin{itemize}
\item \emph{Mordukhovich stationary (M-stationary)} if $-\nabla f(x) \in \norcone{C}{x}$;
\item \emph{Bouligand stationary (B-stationary)} if $-\nabla f(x) \in \regnorcone{C}{x}$;
\item \emph{proximally stationary (P-stationary)} if $-\nabla f(x) \in \proxnorcone{C}{x}$.
\end{itemize}
\end{definition}

There are many practical examples of a set $C$ for which at least one of the inclusions in~\eqref{eq:NestedNormalCones} is strict, especially the second one. This is illustrated by the four examples studied in section~\ref{sec:ExamplesFeasibleSets}, where the second inclusion is strict at infinitely many points. The three notions of stationarity are therefore not equivalent.

For problem~\eqref{eq:MinDiffFunctionClosedSet}, B-stationarity and P-stationarity are the strongest necessary conditions for local optimality under hypotheses \ref{it:continuously_diff} and \ref{it:Lipschitz_diff}, respectively. This is a consequence of the following gradient characterizations of the regular and proximal normal cones given in \cite[Theorem~6.11]{RockafellarWets} and Theorem~\ref{thm:GradientCharacterizationProximalNormals}: for all $x \in C$,
\begin{align}
\regnorcone{C}{x}
\label{eq:GradientCharacterizationRegularNormals}
&= \left\{-\nabla h(x) ~\Big| \begin{array}{l} h : \mathcal{E} \to \R \text{ satisfies~\ref{it:continuously_diff}},\\ x \text{ is a local minimizer of } h|_C \end{array}\right\},\\
\proxnorcone{C}{x}
\label{eq:GradientCharacterizationProximalNormals}
&= \left\{-\nabla h(x) ~\Big| \begin{array}{l} h : \mathcal{E} \to \R
\text{ satisfies~\ref{it:Lipschitz_diff}},\\
x \text{ is a local minimizer of } h|_C \end{array}\right\}.
\end{align}
The inclusion $\supseteq$ in~\eqref{eq:GradientCharacterizationRegularNormals} shows that, under \ref{it:continuously_diff}, every local minimizer of $f|_C$ is B-stationary for~\eqref{eq:MinDiffFunctionClosedSet}. Thus, under \ref{it:continuously_diff}, $\regnorcone{C}{x}$ is sufficiently large to yield a necessary condition for local optimality. The inclusion $\subseteq$ in~\eqref{eq:GradientCharacterizationRegularNormals} shows that, under \ref{it:continuously_diff}, replacing $\regnorcone{C}{x}$ with one of its proper subsets would yield a condition that is not necessary for local optimality. The two inclusions in~\eqref{eq:GradientCharacterizationProximalNormals} yield the corresponding facts for P-stationarity and the proximal normal cone under~\ref{it:Lipschitz_diff}.

In comparison, M-stationarity is a weaker notion of stationarity, which is considered unsatisfactory in \cite[\S 4]{HosseiniLukeUschmajew2019}, \cite[\S 1]{LevinKileelBoumal2023}, and \cite[\S 2.1]{Pauwels}. Furthermore, as explained in~\cite{LevinKileelBoumal2023}, distinguishing convergence to a B-stationary point from convergence to an M-stationary point is difficult (a phenomenon formalized by the notion of \emph{apocalypse} in~\cite{LevinKileelBoumal2023}) in the sense that it cannot be done based on standard measures of B-stationarity because of their possible lack of lower semicontinuity at points where the feasible set is not Clarke regular, as also explained in \cite[\S\S 1.1 and 2.4]{OlikierGallivanAbsil2024}.

Projected gradient descent, or $\pgd$ for short, is a basic algorithm aiming at solving problem~\eqref{eq:MinDiffFunctionClosedSet}. To the best of our knowledge, the first article to have considered $\pgd$ on a possibly nonconvex closed set was \cite{BeckTeboulle}. The nonmonotone backtracking version considered in this paper is defined as Algorithm~\ref{algo:PGD} and is based on \cite[Algorithm~3.1]{JiaEtAl} and \cite[Algorithm~3.1]{DeMarchi}. Given $x \in C$ as input, the iteration map of $\pgd$, called the $\pgd$ map and defined as Algorithm~\ref{algo:PGDmap}, performs a backtracking projected line search along the direction of $-\nabla f(x)$: it computes an arbitrary projection $y$ of $x-\alpha\nabla f(x)$ onto $C$ for decreasing values of the step size $\alpha \in (0, \infty)$ until $y$ satisfies an Armijo condition.
In the simplest version of $\pgd$, called \emph{monotone}, the Armijo condition ensures that the value of $f$ at the next iterate is smaller by a specified amount than the value at the current iterate. Following the general settings proposed in \cite{JiaEtAl,KanzowMehlitz} and \cite{DeMarchi}, the value at the current iterate can be replaced with the maximum value of $f$ over a prefixed number of the previous iterates (``max'' rule) or with a weighted average of the values of $f$ at the previous iterates (``average'' rule). This version of $\pgd$ is called \emph{nonmonotone}.
Monotone $\pgd$ accumulates at M-stationary points of~\eqref{eq:MinDiffFunctionClosedSet} if $f$ is continuously differentiable on~$\mathcal{E}$ and bounded from below on~$C$ \cite[Theorem~3.1]{KanzowMehlitz}. The same result holds for nonmonotone $\pgd$ with the ``average'' rule \cite[Theorem~4.6]{DeMarchi}, and also for nonmonotone $\pgd$ with the ``max'' rule if $f$ is further uniformly continuous on the sublevel set
\begin{equation}
\label{eq:InitialSublevelSet}
\{x \in C \mid f(x) \le f(x_0)\},
\end{equation}
where $x_0 \in C$ is the initial iterate given to the algorithm \cite[Theorem~4.1]{KanzowMehlitz}.
However, as pointed out in \cite[\S 1]{LevinKileelBoumal2023}, it is an open question whether the accumulation points of $\pgd$ are always B-stationary for~\eqref{eq:MinDiffFunctionClosedSet}.

This paper answers positively the question by proving Theorem~\ref{thm:StationarityAccumulationPointsPGD}.

\begin{theorem}
\label{thm:StationarityAccumulationPointsPGD}
Consider a sequence generated by $\pgd$ (Algorithm~\ref{algo:PGD}) when applied to problem~\eqref{eq:MinDiffFunctionClosedSet}.
\begin{itemize}
\item If this sequence is finite, then its last element is B-stationary for~\eqref{eq:MinDiffFunctionClosedSet} under~\ref{it:continuously_diff}, and even P-stationary for~\eqref{eq:MinDiffFunctionClosedSet} under~\ref{it:Lipschitz_diff}.
\item If this sequence is infinite, then all of its accumulation points, if any, are B-stationary for~\eqref{eq:MinDiffFunctionClosedSet} under~\ref{it:continuously_diff}, and even P-stationary for~\eqref{eq:MinDiffFunctionClosedSet} under~\ref{it:Lipschitz_diff}.
\end{itemize}
\end{theorem}

If $\nabla f$ is globally Lipschitz continuous, then it is known that every local minimizer of $f|_C$ is P-stationary for~\eqref{eq:MinDiffFunctionClosedSet} \cite[Proposition~3.5(ii)]{ThemelisStellaPatrinos} (the result is given for a global minimizer but the proof shows that it also holds for a local minimizer) and that $\pgd$ with a constant step size smaller than the inverse of the Lipschitz constant accumulates at P-stationary points of~\eqref{eq:MinDiffFunctionClosedSet} \cite[Theorem~5.6(i)]{ThemelisStellaPatrinos}. Indeed, the ZeroFPR algorithm proposed in \cite{ThemelisStellaPatrinos} extends the proximal gradient algorithm with a constant step size \cite[Remark~5.5]{ThemelisStellaPatrinos}, which itself extends $\pgd$ with a constant step size; problem~\eqref{eq:MinDiffFunctionClosedSet} corresponds to \cite[problem~(1.1)]{ThemelisStellaPatrinos} with $g$ the indicator function of our set $C$. These results were rediscovered in \cite{Pauwels} where, in addition, the distance from the negative gradient of the continuously differentiable function to the regular subdifferential of the other function is proven to converge to zero along the generated sequence, and a quadratic lower bound on $\varphi-\varphi(\oshort{x})$ for every accumulation point $\oshort{x}$ is obtained, where $\varphi$ denotes the sum of the two functions. The two results cited from \cite{ThemelisStellaPatrinos} were already stated in \cite[Theorems~2.2 and~3.1]{BeckEldar} for $C$ the set $\sparse{n}{s}$ of vectors of $\R^n$ having at most $s$ nonzero components for some positive integer $s < n$, and in \cite[Proposition~1 and Theorem~1]{BalashovPolyakTremba} for $C$ satisfying a regularity condition called \emph{proximal smoothness}, which none of the four examples studied in section~\ref{sec:ExamplesFeasibleSets} satisfies.

This paper is organized as follows. The necessary background in variational analysis is introduced in section~\ref{sec:ElementsOfVariationalAnalysis}. The literature on stationarity notions for problem~\eqref{eq:MinDiffFunctionClosedSet} is surveyed in section~\ref{sec:StationarityLiterature}. The $\pgd$ algorithm is reviewed in section~\ref{sec:PGD}. It is analyzed under hypothesis~\ref{it:continuously_diff} in section~\ref{sec:ConvergenceAnalysisForContinuousGradient} and under hypothesis~\ref{it:Lipschitz_diff} in section~\ref{sec:ConvergenceAnalysisForLocallyLipschitzContinuousGradient}. Four practical examples of a set $C$ for which the first inclusion in~\eqref{eq:NestedNormalCones} is an equality for all $x \in C$ and the second is strict for infinitely many $x \in C$ are given in section~\ref{sec:ExamplesFeasibleSets}. Theorem~\ref{thm:StationarityAccumulationPointsPGD} is illustrated by a comparison between $\pgd$ and a first-order algorithm that is not guaranteed to accumulate at B-stationary points of~\eqref{eq:MinDiffFunctionClosedSet} in section~\ref{sec:PGDvsP2GD}. Concluding remarks are gathered in section~\ref{sec:Conclusion}.

\section{Elements of variational analysis}
\label{sec:ElementsOfVariationalAnalysis}
This section, mostly based on \cite{RockafellarWets}, reviews background material in variational analysis that is used in the rest of the paper. Section~\ref{subsec:ProjectionMap} concerns the projection map onto $C$. Section~\ref{subsec:NormalityStationarity} reviews the three notions of normality on which the three notions of stationarity provided in Definition~\ref{def:M/B/P-Stationarity} are based.

Recall that, throughout the paper, $\mathcal{E}$ is a Euclidean vector space and $C \subseteq \mathcal{E}$ is nonempty and closed. Moreover, for every $x \in \mathcal{E}$ and $\rho \in (0, \infty)$, $\ball(x, \rho) \coloneq \{y \in \mathcal{E} \mid \norm{x-y} < \rho\}$ and $\ball[x, \rho] \coloneq \{y \in \mathcal{E} \mid \norm{x-y} \le \rho\}$ are respectively the open and closed balls of center $x$ and radius $\rho$ in~$\mathcal{E}$.
A nonempty subset $K$ of $\mathcal{E}$ is called a \emph{cone} if $x \in K$ implies $\alpha x \in K$ for all $\alpha \in [0, \infty)$ \cite[\S 3B]{RockafellarWets}.

\subsection{Projection map}
\label{subsec:ProjectionMap}
Given $x \in \mathcal{E}$, the distance from $x$ to $C$ is $\dist(x, C) \coloneq \min_{y \in C} \norm{x-y}$ and the projection of $x$ onto $C$ is $\proj{C}{x} \coloneq \argmin_{y \in C} \norm{x-y}$. The function $\mathcal{E} \to \R : x \mapsto \dist(x, C)$ is continuous and, for every $x \in \mathcal{E}$, the set $\proj{C}{x}$ is nonempty and compact \cite[Example~1.20]{RockafellarWets}.
Proposition~\ref{prop:ProjectedTranslation} is invoked frequently in the rest of the paper.

\begin{proposition}
\label{prop:ProjectedTranslation}
For all $x \in C$, $v \in \mathcal{E}$, and $y \in \proj{C}{x-v}$,
\begin{align}
\label{eq:ProjectedTranslationDistance}
\norm{y-x} &\le 2 \norm{v},\\
\label{eq:ProjectedTranslationIP}
2 \ip{v}{y-x} &\le - \norm{y-x}^2,
\end{align}
and the inequalities are strict if $x \notin \proj{C}{x-v}$.
\end{proposition}

\begin{proof}
By definition of the projection, $\norm{y-(x-v)} \le \norm{x-(x-v)} = \norm{v}$ and the inequality is strict if $x \notin \proj{C}{x-v}$. Thus, on the one hand,
\begin{equation*}
\norm{y-x}
= \norm{y-(x-v)-v}
\le \norm{y-(x-v)} + \norm{-v}
\le \norm{v} + \norm{v}
= 2 \norm{v},
\end{equation*}
and, on the other hand, $\norm{y-(x-v)}^2 \le \norm{v}^2$, which is equivalent to~\eqref{eq:ProjectedTranslationIP}.
\end{proof}

\subsection{Normality and stationarity}
\label{subsec:NormalityStationarity}
Based on \cite[Chapter~6]{RockafellarWets}, this section reviews the three notions of normality on which the three notions of stationarity given in Definition~\ref{def:M/B/P-Stationarity} are based.

A vector $v \in \mathcal{E}$ is said to be \emph{tangent} to $C$ at $x \in C$ if there exist sequences $(x_i)_{i \in \N}$ in $C$ converging to $x$ and $(t_i)_{i \in \N}$ in $(0, \infty)$ such that the sequence $(\frac{x_i-x}{t_i})_{i \in \N}$ converges to $v$ \cite[Definition~6.1]{RockafellarWets}. The set of all tangent vectors to $C$ at $x \in C$ is a closed cone \cite[Proposition~6.2]{RockafellarWets} called the \emph{tangent cone} to $C$ at $x$ and denoted by $\tancone{C}{x}$.
The \emph{regular normal cone} to $C$ at $x \in C$ is
\begin{equation*}
\regnorcone{C}{x} \coloneq \left\{v \in \mathcal{E} \mid \ip{v}{w} \le 0 \; \forall w \in \tancone{C}{x}\right\},
\end{equation*}
which is a closed convex cone \cite[Definition~6.3 and Proposition~6.5]{RockafellarWets}.
A vector $v \in \mathcal{E}$ is said to be \emph{normal (in the general sense)} to $C$ at $x \in C$ if there exist sequences $(x_i)_{i \in \N}$ in $C$ converging to $x$ and $(v_i)_{i \in \N}$ converging to $v$ such that, for all $i \in \N$, $v_i \in \regnorcone{C}{x_i}$ \cite[Definition~6.3]{RockafellarWets}. The set of all normal vectors to $C$ at $x \in C$ is a closed cone \cite[Proposition~6.5]{RockafellarWets} called the \emph{normal cone} to $C$ at $x$ and denoted by $\norcone{C}{x}$.
A vector $v \in \mathcal{E}$ is called a \emph{proximal normal} to $C$ at $x \in C$ if there exists $\oshort{\alpha} \in (0, \infty)$ such that $x \in \proj{C}{x+\oshort{\alpha}v}$, i.e., $\oshort{\alpha}\norm{v} = \dist(x+\oshort{\alpha}v, C)$, which implies that, for all $\alpha \in [0, \oshort{\alpha})$, $\proj{C}{x+\alpha v} = \{x\}$ \cite[Example~6.16]{RockafellarWets}. The set of all proximal normals to $C$ at $x \in C$ is a convex cone called the \emph{proximal normal cone} to $C$ at $x$ and denoted by $\proxnorcone{C}{x}$.

As stated in~\eqref{eq:NestedNormalCones}, for all $x \in C$,
\begin{equation*}
\proxnorcone{C}{x} \subseteq \regnorcone{C}{x} \subseteq \norcone{C}{x}.
\end{equation*}
The closed set $C$ is said to be \emph{Clarke regular} at $x \in C$ if $\regnorcone{C}{x} = \norcone{C}{x}$ \cite[Definition~6.4]{RockafellarWets}. Thus, M-stationarity is equivalent to B-stationarity at a point $x \in C$ if and only if $C$ is Clarke regular at $x$, which is not the case in many practical situations, as illustrated by the four examples given in section~\ref{sec:ExamplesFeasibleSets}. For those examples, however, regular normals are proximal normals (Proposition~\ref{prop:ProximalRegularLowSparsityOrRank}). An example of a set $C$ and a point $x \in C$ such that both inclusions in~\eqref{eq:NestedNormalCones} are strict is given in Example~\ref{example:DifferentNormalCones}.

\begin{example}
\label{example:DifferentNormalCones}
Let $\mathcal{E} \coloneq \R^2$ and $C \coloneq \left\{\left(t, \max\left\{0, t^{3/5}\right\}\right) \mid t \in \R\right\}$ (inspired by \cite[Figure~6--12(a)]{RockafellarWets}). Then,
\begin{align*}
\tancone{C}{0, 0} &= (\{0\} \times [0, \infty)) \cup ((-\infty, 0] \times \{0\}),\\
\regnorcone{C}{0, 0} &= [0, \infty) \times (-\infty, 0],\\
\proxnorcone{C}{0, 0} &= \regnorcone{C}{0, 0} \setminus ((0, \infty) \times \{0\}),\\
\norcone{C}{0, 0} &= \regnorcone{C}{0, 0} \cup \tancone{C}{0, 0}.
\end{align*}
Thus,
\begin{equation*}
\proxnorcone{C}{0, 0}
\subsetneq \regnorcone{C}{0, 0}
\subsetneq \norcone{C}{0, 0}.
\end{equation*}
This is illustrated in Figure~\ref{fig:DifferentNormalCones}.
\begin{figure}[h]
\begin{center}
\begin{tikzpicture}[scale=2]
\def\r{0.5}
%Set
\draw (-1, 0) -- (0, 0);
%Normal cone at (0, 0)
\draw [thick, blue!30] (-\r, 0) -- (0, 0) -- (0, \r);
\fill [JungleGreen!60] (0, -\r) arc (270:360:\r);
\fill [JungleGreen!60] (0, 0) -- (0, -\r) -- (\r, 0) -- cycle;
\draw [JungleGreen!60] (0.003, {0.003-\r}) -- ({\r-0.003}, -0.003);
\draw [thick, red] (0, 0) -- (\r, 0);
\draw (0, {-1.5*\r}) node {$\norcone{C}{0, 0} = {\color{JungleGreen!60} \proxnorcone{C}{0, 0}} \cup {\color{red} (\regnorcone{C}{0, 0} \setminus \proxnorcone{C}{0, 0})} \cup {\color{blue!30} \tancone{C}{0, 0}}$};
%Set
\draw [domain=0:1, samples=400] plot (\x, {\x^0.6});
%Set and point
\draw (0, 0) node {{\tiny $\bullet$}};
\draw ({-0.4*\r}, {0.22*\r}) node {$(0, 0)$};
\draw (1.1, 1) node {$C$};
\end{tikzpicture}
\end{center}
\caption{Tangent and normal cones from Example~\ref{example:DifferentNormalCones}.}
\label{fig:DifferentNormalCones}
\end{figure}
\end{example}

As pointed out in section~\ref{sec:Introduction}, the regular and proximal normal cones enjoy gradient characterizations that imply that, for problem~\eqref{eq:MinDiffFunctionClosedSet}, B- and P-stationarity are the strongest necessary conditions for local optimality under \ref{it:continuously_diff} and \ref{it:Lipschitz_diff}, respectively. The characterization given in~\eqref{eq:GradientCharacterizationRegularNormals} follows from \cite[Theorem~6.11]{RockafellarWets}. That given in~\eqref{eq:GradientCharacterizationProximalNormals} comes from Theorem~\ref{thm:GradientCharacterizationProximalNormals}, established at the end of this section.
Example~\ref{example:ProxNorConeNotNecessary} illustrates that, for problem~\eqref{eq:MinDiffFunctionClosedSet} under \ref{it:continuously_diff}, P-stationarity is not necessary for local optimality. 

\begin{example}
\label{example:ProxNorConeNotNecessary}
Let $\mathcal{E} \coloneq \R^2$, $C \coloneq \left\{(x_1, x_2) \in \R^2 \mid x_2 \ge \max\left\{0, x_1^{3/5}\right\}\right\}$ \cite[Figure~6--12(a)]{RockafellarWets}, and $f : \R^2 \to \R : (x_1, x_2) \mapsto \frac{1}{2}(x_1-1)^2 + |x_2|^{3/2}$. Then, $f$ is continuously differentiable on $\mathcal{E}$, hence on $C$, and, for all $(x_1, x_2) \in \R^2$, $\nabla f(x_1, x_2) = (x_1-1, \frac{3}{2}\mathrm{sgn}(x_2)|x_2|^{1/2})$. Thus, $-\nabla f(0, 0) = (1, 0) \in \regnorcone{C}{0, 0} \setminus \proxnorcone{C}{0, 0}$, yet $\argmin_C f = \{(0, 0)\}$. 
\end{example}

Proposition~\ref{prop:ProxNorConeNecessary} states that P-stationarity is necessary for local optimality if $f$ is assumed to satisfy~\ref{it:Lipschitz_diff}, that is, $f$
is differentiable on~$\mathcal{E}$ and $\nabla f$ is locally Lipschitz continuous. The latter means that, for every open or closed ball $\mathcal{B} \subsetneq \mathcal{E}$,
\begin{equation*}
\lip_{\mathcal{B}}(\nabla f) \coloneq \sup_{\substack{x, y \in \mathcal{B} \\ x \ne y}} \frac{\norm{\nabla f(x) - \nabla f(y)}}{\norm{x-y}} < \infty,
\end{equation*}
which implies, by \cite[Lemma~1.2.3]{Nesterov2018}, that, for all $x, y \in \mathcal{B}$,
\begin{equation}
\label{eq:InequalityLipschitzContinuousGradient}
|f(y) - f(x) - \ip{\nabla f(x)}{y-x}| \le \frac{\lip_{\mathcal{B}}(\nabla f)}{2} \norm{y-x}^2.
\end{equation}

\begin{proposition}
\label{prop:ProxNorConeNecessary}
Assume that $f$ satisfies~\ref{it:Lipschitz_diff}. If $x \in C$ is a local minimizer of $f|_C$, then $-\nabla f(x) \in \proxnorcone{C}{x}$.
\end{proposition}

\begin{proof}
By contrapositive. Assume that $-\nabla f(x) \notin \proxnorcone{C}{x}$ for some $x \in C$. Let $\rho \in (0, \infty)$. Then, for all $\alpha \in (0, \frac{\rho}{2\norm{\nabla f(x)}}]$,
\begin{equation*}
x
\notin \proj{C}{x-\alpha\nabla f(x)}
\subseteq \ball(x, 2\alpha\norm{\nabla f(x)})
\subseteq \ball(x, \rho),
\end{equation*}
where the first inclusion holds by \eqref{eq:ProjectedTranslationDistance}. Thus, by~\eqref{eq:InequalityLipschitzContinuousGradient} and~\eqref{eq:ProjectedTranslationIP}, for all $\alpha \in (0, \min\{\frac{\rho}{2\norm{\nabla f(x)}}, \frac{1}{\lip_{\ball(x, \rho)}(\nabla f)}\}]$ and $y \in \proj{C}{x-\alpha\nabla f(x)}$,
\begin{align*}
f(y)-f(x)
&\le \ip{\nabla f(x)}{y-x} + \frac{\lip_{\ball(x, \rho)}(\nabla f)}{2} \norm{y-x}^2\\
&< \left(-\frac{1}{2\alpha} + \frac{\lip_{\ball(x, \rho)}(\nabla f)}{2}\right) \norm{y-x}^2\\
&\le 0.
\end{align*}
Hence, $x$ is not a local minimizer of $f|_C$.
\end{proof}

Theorem~\ref{thm:GradientCharacterizationProximalNormals} strengthens \cite[Proposition~8.46(d)]{RockafellarWets} by stating that \eqref{eq:GradientCharacterizationProximalNormals} is valid.

\begin{theorem}[gradient characterization of proximal normals]
\label{thm:GradientCharacterizationProximalNormals}
For every $x \in C$, \eqref{eq:GradientCharacterizationProximalNormals} holds.
\end{theorem}

\begin{proof}
Let $x \in C$. The inclusion $\supseteq$ holds by Proposition~\ref{prop:ProxNorConeNecessary}.
For the inclusion~$\subseteq$, let $v \in \proxnorcone{C}{x}$. By definition of $\proxnorcone{C}{x}$, there exists $\oshort{\alpha} \in (0, \infty)$ such that $x \in \proj{C}{x+\oshort{\alpha}v}$. This is equivalent to the fact that $x$ is a global minimizer of $h|_C$ with
\begin{equation*}
h : \mathcal{E} \to \R : y \mapsto \frac{1}{2\oshort{\alpha}} \norm{y - (x+\oshort{\alpha}v)}^2.
\end{equation*}
The function $h$ is differentiable, its gradient is locally Lipschitz continuous (actually, globally Lipschitz continuous, since it is an affine map), and $-\nabla h(x) = v$.
Since $x$ is a global minimizer of $h|_C$, it is also a local minimizer of $h|_C$. Thus,
\begin{equation*}
v \in \left\{-\nabla h(x) ~\Big|
\begin{array}{l} h : \mathcal{E} \to \R \text{ satisfies~\ref{it:Lipschitz_diff}},
\\ x \text{ is a local minimizer of } h|_C \end{array}\right\},
\end{equation*}
which implies the inclusion $\subseteq$ in~\eqref{eq:GradientCharacterizationProximalNormals}.
\end{proof}

\begin{remark}
From our proof, we see that \eqref{eq:GradientCharacterizationProximalNormals} is also true if we replace ``local minimizer'' with ``global minimizer''. The same holds for~\eqref{eq:GradientCharacterizationRegularNormals} \cite[Theorem~6.11]{RockafellarWets}. However, in this section, we are interested in understanding the closeness between the notions of stationarity and \textit{local} optimality.
\end{remark}

\section{Stationarity in the literature}
\label{sec:StationarityLiterature}
This section surveys the names given to the stationarity notions provided in Definition~\ref{def:M/B/P-Stationarity} and attempts to offer a brief historical perspective.
The terms ``B-stationarity'' and ``M-stationarity'' first appeared in the literature about mathematical programs with equilibrium constraints (MPECs), as explained in sections~\ref{subsec:HistoryBouligandStationarity} and \ref{subsec:HistoryMordukhovichStationarity}.
In contrast, the term ``P-stationarity'' seems to be new. P-stationarity is called ``criticality'' in \cite[Definition~3.1(ii)]{ThemelisStellaPatrinos}; recall that problem~\eqref{eq:MinDiffFunctionClosedSet} is \cite[problem~(1.1)]{ThemelisStellaPatrinos} with $g$ the indicator function of our set $C$. We propose the name ``P-stationarity'' because this stationarity notion is based on the proximal normal cone. It is related to the so-called $\alpha$-stationarity; see section~\ref{subsec:ProximalStationarityLiterature}.

\subsection{A brief history of B-stationarity}
\label{subsec:HistoryBouligandStationarity}
Peano knew that B-stationarity is a necessary condition for optimality. The statement is implicit in his 1887 book \emph{Applicazioni geometriche del calcolo infinitesimale} and explicit in his 1908 book \emph{Formulario Mathematico}, where the formulation is based on the tangent cone and the derivative defined in the same book; see the historical investigation in \cite{DoleckiGreco2007,DoleckiGreco2011}.

B-stationarity appears as a necessary condition for optimality in \cite[Theorem~2.1]{Varaiya} and \cite[Theorem~1]{Guignard}, without any reference to Peano's work. The latter theorem uses the polar of the closure of the convex hull of the tangent cone, which equals the polar of the tangent cone by \cite[Corollary~6.21]{RockafellarWets}. Neither ``stationary'' nor ``critical'' appears in \cite{Varaiya} or \cite{Guignard}.

The ``Bouligand derivative'', or ``B-derivative'' for short, was introduced in \cite{Robinson1987}. It is a special case of the contingent derivative introduced by Aubin based on the tangent cone. The name ``Bouligand derivative'' was chosen because the tangent cone is generally attributed to Bouligand; see, e.g., \cite{RockafellarWets,Mordukhovich1,Mordukhovich2018} for recent references. Differentiability implies B-differentiability.

In \cite[\S 4]{ScholtesStohr}, a point where a real-valued function is B-differentiable is called a ``Bouligand stationary (B-stationary) point'' of the function if the B-derivative at that point is nonnegative. This is a stationarity concept for unconstrained optimization, which therefore does not apply to problem~\eqref{eq:MinDiffFunctionClosedSet}.

B-stationarity is called a ``stationarity condition'' and said to be ``well known'' in \cite[\S 4.1]{LuoPangRalphWu} where \cite{Guignard} is cited.

In \cite[\S 2.1]{ScheelScholtes}, the term ``B-stationarity'' is used to name the stationarity concept for an MPEC that corresponds to B-stationarity in the sense of \cite[\S 4]{ScholtesStohr} for a nonsmooth reformulation of the MPEC \cite[Proposition~6]{ScheelScholtes}.
As pointed out in \cite[\S 2.1]{Ye2005} and \cite[\S 3.3]{FlegelKanzow2005JOTA}, B-stationarity in the sense of \cite[\S 2.1]{ScheelScholtes}, which is specific to MPECs, is not B-stationarity in the sense of Definition~\ref{def:M/B/P-Stationarity} and is called ``MPEC-linearized B-stationarity'' in \cite[\S 3.3]{FlegelKanzow2005JOTA} to avoid confusion. Nevertheless, this MPEC-linearized B-stationarity appears under the name ``B-stationarity'' in \cite[\S 1.1]{HuRalph}, \cite[Definition~2.2]{GuoLin}, and \cite[Definition~3.2]{WuZhangZhang}, which all cite \cite{ScheelScholtes}.

The term ``B-stationarity'' was used to name the absence of descent directions in the tangent cone (as in Definition~\ref{def:M/B/P-Stationarity}) first in \cite[\S 1]{PangFukushima}. It was used in this sense in several subsequent works by various authors; see, e.g., \cite[\S 2]{FukushimaPang}, \cite[\S 2]{FukushimaTseng}, \cite[Definition~2.4]{FukushimaLin}, \cite[Definition~2.2]{Ye2005}, \cite[\S\S 3.3 and 4]{FlegelKanzow2005JOTA}, \cite[\S 3]{FlegelKanzow2005Optimization}, \cite[\S 2]{Pang}, \cite[Definition~2.4]{SteffensenUlbrich}, \cite[Definition 3.4]{Gfrerer}, \cite[(18)]{PangRazaviyaynAlvarado}, \cite[Definition~3(1)]{BenkoGfrerer2017}, \cite[Definition~4(i)]{BenkoGfrerer2018}, \cite[\S 4]{HosseiniLukeUschmajew2019}, and \cite[Definition~6.1.1]{CuiPang}.

In \cite[Definition~6.1.1]{CuiPang}, B-stationarity is defined for the problem of minimizing a real-valued function that is B-differentiable on a nonempty closed subset of a Euclidean vector space, thereby extending the concept introduced in \cite[\S 4]{ScholtesStohr} to constrained optimization. This more general definition reduces to that from Definition~\ref{def:M/B/P-Stationarity} if the function is differentiable.

B-stationarity is also known under other names in the literature.
First, in \cite[Definition~1(b)]{FlegelKanzowOutrata} and \cite[\S 3]{Mehlitz}, B-stationarity is called ``strong stationarity''; \cite[problem~(4)]{FlegelKanzowOutrata} and \cite[(P2)]{Mehlitz} reduce to problem~\eqref{eq:MinDiffFunctionClosedSet} for $F$ the identity map on $\R^n$.
Second, because the regular normal cone is also called the Fréchet normal cone, especially in infinite-dimensional spaces \cite{RockafellarWets,Mordukhovich1,Mordukhovich2018}, B-stationarity is called ``Fréchet stationarity'', or ``F-stationarity'' for short, in \cite[Definition~4.1(ii)]{LiSongXiu2019}, \cite[Definition~5.1(i)]{LiXiuZhou2020}, \cite[Definition~3.2(ii)]{LuoQi2023}, and \cite{Pauwels}.
Third, B-stationarity is simply called ``stationarity'' (or ``criticality'') in \cite[\S 2.1]{SchneiderUschmajew2015}, \cite[\S 2.1.1]{HaLiuBarber2020}, \cite[Definition~2.3]{LevinKileelBoumal2023}, \cite[Definition~3.2(c)]{LevinKileelBoumal2025}, and \cite[Definition~1]{GaoPengYuan}.

\subsection{A brief history of M-stationarity}
\label{subsec:HistoryMordukhovichStationarity}
According to \cite[\S 2]{FlegelKanzowOutrata}, the term ``M-stationarity'' was introduced in \cite{Scholtes} for an MPEC. This name was chosen because the corresponding stationarity condition was derived from the generalized differential calculus of Mordukhovich. To the best of our knowledge, the term ``M-stationarity'' was used to indicate that the negative gradient is in the normal cone (as in Definition~\ref{def:M/B/P-Stationarity}) first in \cite[Definition~1(a)]{FlegelKanzowOutrata}; recall that \cite[problem~(4)]{FlegelKanzowOutrata} reduces to problem~\eqref{eq:MinDiffFunctionClosedSet} for $F$ the identity map on $\R^n$. There, the name is motivated by the presence of the normal cone that was introduced by Mordukhovich. M-stationarity appears, under this name, in several subsequent works by various authors; see, e.g., \cite[Definition~3(3)]{BenkoGfrerer2017}, \cite[Definition~4(iii)]{BenkoGfrerer2018}, \cite[\S 4]{HosseiniLukeUschmajew2019}, \cite[\S 3]{Mehlitz}, \cite[\S 2]{KanzowMehlitz}, \cite[\S 3]{JiaEtAl}, and \cite[\S 2.3]{JiaKanzowMehlitz}.

\subsection{P-stationarity and $\alpha$-stationarity}
\label{subsec:ProximalStationarityLiterature}
P-stationarity is closely related to \mbox{$\alpha$-stationarity}, which was introduced in \cite[Definition~2.3]{BeckEldar} for $C = \sparse{n}{s}$ and in \cite[Definition~4.1(i)]{LiSongXiu2019}, \cite[\S 2.1.1]{HaLiuBarber2020}, \cite[Definition~5.1(ii)]{LiXiuZhou2020}, \cite[(4.2)]{LiLuo2023}, and \cite[Definition~3.2(i)]{LuoQi2023} for several sets of low-rank matrices. By definition of the proximal normal cone, a point $x \in C$ is P-stationary for~\eqref{eq:MinDiffFunctionClosedSet} if and only if there exists $\alpha \in (0, \infty)$ such that $x \in \proj{C}{x-\alpha\nabla f(x)}$. In contrast, given $\alpha \in (0, \infty)$, a point $x \in C$ is said to be $\alpha$-stationary for~\eqref{eq:MinDiffFunctionClosedSet} if $x \in \proj{C}{x-\alpha\nabla f(x)}$. Thus, while $\alpha$-stationarity prescribes the number $\alpha \in (0, \infty)$, P-stationarity merely requires the existence of such a number. Furthermore, $\alpha$-stationarity should not be confused with the approximate stationarity from \cite[Definition~2.6]{LevinKileelBoumal2023}.

\section{The PGD algorithm}
\label{sec:PGD}
This section reviews the $\pgd$ algorithm, as defined in \cite[Algorithm~3.1]{JiaEtAl} except that the ``average'' rule is allowed as an alternative to the ``max'' rule. Its iteration map, called the $\pgd$ map, is defined as Algorithm~\ref{algo:PGDmap}. $\pgd$ is defined as Algorithm~\ref{algo:PGD}, which uses Algorithm~\ref{algo:PGDmap} as a subroutine. The nonmonotonic behavior of $\pgd$ is described in Propositions~\ref{prop:NonmonotonePGD} and~\ref{prop:NonmonotonePGDAverage}.

\begin{algorithm}[H]
\caption{$\pgd$ map}
\label{algo:PGDmap}
\begin{algorithmic}[1]
\Require
$(\mathcal{E}, C, f, \ushort{\alpha}, \oshort{\alpha}, \beta, c)$, where $\mathcal{E}$ is a Euclidean vector space, $C$ is a nonempty closed subset of~$\mathcal{E}$, $f : \mathcal{E} \to \R$ satisfies \ref{it:continuously_diff}, $0 < \ushort{\alpha} \le \oshort{\alpha} < \infty$, and $\beta, c \in (0, 1)$.
\Input
$(x, \mu)$ with $x \in C$ and $\mu \in [f(x), \infty)$.
\Output
$y \in \pgd(x, \mu; \mathcal{E}, C, f, \ushort{\alpha}, \oshort{\alpha}, \beta, c)$.

\State
Choose $\alpha \in [\ushort{\alpha}, \oshort{\alpha}]$ and $y \in \proj{C}{x-\alpha\nabla f(x)}$;
\label{algo:PGDmap:InitialStepSize}
\While
{$f(y) > \mu + c \ip{\nabla f(x)}{y-x}$}
\label{algo:PGDmap:NonmonotoneArmijoCondition}
\State
$\alpha \gets \alpha \beta$;
\State
Choose $y \in \proj{C}{x-\alpha\nabla f(x)}$;
\EndWhile
\State
Return $y$.
\end{algorithmic}
\end{algorithm}

\begin{remark}
The Armijo condition
\begin{equation*}
f(y) \le \mu + c \ip{\nabla f(x)}{y-x}
\end{equation*}
ensures that the decrease $\mu-f(y)$ is at least a fraction $c$ of the opposite of the directional derivative of $f$ at $x$ along the update vector $y-x$. By \eqref{eq:ProjectedTranslationIP}, this condition implies that
\begin{equation}\label{eq:KanzowMehlitzDecay}
f(y) \le \mu - \frac{c}{2\alpha} \norm{y-x}^2,
\end{equation}
which is the condition used in \cite[Algorithms~3.1 and 4.1]{KanzowMehlitz} and \cite[Algorithm~3.1]{DeMarchi}. Importantly, all results from \cite{KanzowMehlitz} hold for both conditions, as is clear from the proofs.
% Par contre, les résultats de De Marchi ne sont pas clairement valides avec notre condition (je pense au lemme 4.4).
\end{remark}

\begin{remark}
\label{rem:PGDmap_terminates}
By Proposition~\ref{prop:AcceptedStep}, under \ref{it:continuously_diff}, if $x$ is not B-stationary for~\eqref{eq:MinDiffFunctionClosedSet}, then the while loop in Algorithm~\ref{algo:PGDmap} is guaranteed to terminate, thereby producing a point $y$ such that $f(y) < \mu$; $y \ne x$ holds because $x$ is not B-stationary and hence not P-stationary. If $f$ satisfies~\ref{it:Lipschitz_diff}, then the while loop is guaranteed to terminate for every $x \in C$, by Corollary~\ref{coro:PGDmapTerminationLipschitzContinuousGradient}.
\end{remark}

The $\pgd$ algorithm is defined as Algorithm~\ref{algo:PGD}. It is said to be monotone or nonmonotone depending on whether $\mu_i=f(x_i)$ for all $i$ (that is, $l=0$ for the ``max'' rule, or $p=1$ for the ``average'' rule) or not.

\begin{algorithm}[H]
\caption{$\pgd$}
\label{algo:PGD}
\begin{algorithmic}[1]
\Require
$(\mathcal{E}, C, f, \ushort{\alpha}, \oshort{\alpha}, \beta, c, \text{``nonmonotonicity''})$, where $\mathcal{E}$ is a Euclidean vector space, $C$ is a nonempty closed subset of~$\mathcal{E}$, $f : \mathcal{E} \to \R$ satisfies \ref{it:continuously_diff}, $0 < \ushort{\alpha} \le \oshort{\alpha} < \infty$, $\beta, c \in (0, 1)$, and $\text{``nonmonotonicity''} \in \{(\text{``max''}, l), (\text{``average''}, p)\}$ with $l \in \N$ and $p \in (0, 1]$.
\Input
$x_0 \in C$.
\Output
a sequence in $C$.

\State
$i \gets 0$;
\State
$\mu_{-1} \gets f(x_0)$;
\While
{$-\nabla f(x_i) \notin \regnorcone{C}{x_i}$}
\label{algo:PGD:while}
\If
{$\text{``nonmonotonicity''} = (\text{``max''}, l)$}
\State
$\mu_i \gets \max_{j \in \{\max\{0, i-l\}, \dots, i\}} f(x_j)$;
\ElsIf
{$\text{``nonmonotonicity''} = (\text{``average''}, p)$}
\State
$\mu_i \gets (1-p)\mu_{i-1} + pf(x_i)$; 
\EndIf
\State
Choose $x_{i+1} \in \hyperref[algo:PGDmap]{\pgd}(x_i, \mu_i; \mathcal{E}, C, f, \ushort{\alpha}, \oshort{\alpha}, \beta, c)$;
\label{algo:PGD:PGDmap}
\State
$i \gets i+1$;
\EndWhile
\end{algorithmic}
\end{algorithm}

\begin{remark}
For simplicity, we use a constant weight $p$ in the ``average'' rule. However, we could allow the weight to change from one iteration to the other. It would then be denoted by $p_i$. The main results of the article would hold true in this more general setting, under the additional assumption that $\inf_{i\in\N} p_i >0$.
\end{remark}

\begin{remark}
\label{rem:PGD_H2}
If $f$ satisfies~\ref{it:Lipschitz_diff}, then $\regnorcone{C}{x_i}$ should be replaced with $\proxnorcone{C}{x_i}$ in line~\ref{algo:PGD:while}.
\end{remark}

Examples of a set $C$ for which the projection map and the regular and proximal normal cones can be described explicitly abound; see section~\ref{sec:ExamplesFeasibleSets}. For such examples, Algorithm~\ref{algo:PGD} can be practically implemented.

\begin{remark}
\label{rem:FiniteOrInfinite}
From Remark \ref{rem:PGDmap_terminates}, under \ref{it:continuously_diff}, the call to Algorithm~\ref{algo:PGDmap} in line~\ref{algo:PGD:PGDmap} of $\pgd$ never results in an infinite loop.
Consequently, under \ref{it:continuously_diff}, by running $\pgd$, one always encounters one of the following two situations:
\begin{itemize}
\item $\pgd$ generates a finite sequence, and the last element of this sequence is B-stationary for~\eqref{eq:MinDiffFunctionClosedSet}, and even P-stationary for~\eqref{eq:MinDiffFunctionClosedSet} if $f$ satisfies~\ref{it:Lipschitz_diff}, by Remark~\ref{rem:PGD_H2};
\item $\pgd$ generates an infinite sequence.
\end{itemize}
\end{remark}

The rest of this section and the next two consider the nontrivial case where $\pgd$ generates an infinite sequence. The stationarity of the accumulation points of this sequence, if any, is studied in sections~\ref{sec:ConvergenceAnalysisForContinuousGradient} and \ref{sec:ConvergenceAnalysisForLocallyLipschitzContinuousGradient}. Following \cite[Remark~14]{Polak1971}, which states that it is usually better to determine whether an algorithm generates a sequence having at least one accumulation point by examining the algorithm in the light of the specific problem to which one wishes to apply it, no condition ensuring the existence of a convergent subsequence is imposed. As a reminder, a sequence $(x_i)_{i \in \N}$ in $\mathcal{E}$ has at least one accumulation point if and only if $\liminf_{i \to \infty} \norm{x_i} < \infty$.

A property of monotone $\pgd$ that is helpful for its analysis is the fact that $f$ is strictly decreasing along the generated sequence. For nonmonotone $\pgd$, this is not true. However, weaker properties, stated in the following two propositions, are enough for our purposes.

\begin{proposition}
\label{prop:NonmonotonePGD}
Let $(x_i)_{i \in \N}$ be a sequence generated by $\pgd$ (Algorithm~\ref{algo:PGD}) using the ``max'' rule. For every $i \in \N$, let $g(i) \in \argmax_{j \in \{\max\{0, i-l\}, \dots, i\}} f(x_j)$. Then:
\begin{enumerate}
\item $(f(x_{g(i)}))_{i \in \N}$ is monotonically nonincreasing;
\item $(x_i)_{i \in \N}$ is contained in the sublevel set~\eqref{eq:InitialSublevelSet};
\item if $x \in C$ is an accumulation point of $(x_i)_{i \in \N}$, then $(f(x_{g(i)}))_{i \in \N}$ converges to $\varphi \in [f(x), f(x_0)]$;
\item if $f$ is bounded from below and uniformly continuous on a set that contains $(x_i)_{i \in \N}$, then $(f(x_i))_{i \in \N}$ converges to $\varphi \in \R$.
\end{enumerate}
\end{proposition}

\begin{proof}
The first two statements are \cite[Lemma~4.1 and Corollary~4.1]{KanzowMehlitz}.
For the third one, let $(x_{i_k})_{k \in \N}$ be a subsequence converging to $x$. Since the sequence $(f(x_{g(i)}))_{i\in\N}$ is monotonically nonincreasing, it has a limit in $\R \cup \{-\infty\}$. Thus,
\begin{equation*}
\lim_{i \to \infty} f(x_{g(i)})
= \lim_{k \to \infty} f(x_{g(i_k)})
\ge \lim_{k \to \infty} f(x_{i_k})
= f(x)
> -\infty.
\end{equation*}
It remains to prove the fourth statement. From the first statement and because $f$ is bounded from below, $(f(x_{g(i)}))_{i\in\N}$ converges to some limit $\varphi \in \R$. Assume, for the sake of contradiction, that $(f(x_i))_{i \in \N}$ does not converge to $\varphi$. Then, there exist $\rho \in (0, \infty)$ and a subsequence $(f(x_{i_j}))_{j \in \N}$ contained in $\R \setminus [\varphi-\rho, \varphi+\rho]$. For all $j \in \N$, define $p_j \coloneq g(i_j+l)-i_j \in \{0, \dots, l\}$. Then, there exist $p \in \{0, \dots, l\}$ and a subsequence $(p_{j_k})_{k \in \N}$ such that, for all $k \in \N$, $p_{j_k} = p$. By \cite[(27)]{KanzowMehlitz} or \cite[(A.9)]{JiaEtAl}, $(f(x_{g(i)-p}))_{i \in \N}$ converges to $\varphi$. Therefore, $(f(x_{g(i+l)-p}))_{i \in \N}$ converges to $\varphi$. Hence, $(f(x_{g(i_{j_k}+l)-p}))_{k \in \N}$ converges to $\varphi$. This is a contradiction since, for all $k \in \N$, $f(x_{g(i_{j_k}+l)-p}) = f(x_{i_{j_k}})$.
\end{proof}

\begin{proposition}
\label{prop:NonmonotonePGDAverage}
Let $(x_i)_{i \in \N}$ be a sequence generated by $\pgd$ (Algorithm~\ref{algo:PGD}) using the ``average'' rule. Then:
\begin{enumerate}
\item $(x_i)_{i \in \N}$ is contained in the sublevel set~\eqref{eq:InitialSublevelSet};
\item if  $(x_i)_{i \in \N}$ has an accumulation point, then $(f(x_i))_{i \in \N}$ and $(\mu_i)_{i\in\N}$ converge, toward the same (finite) value.
\end{enumerate}
\end{proposition}

\begin{proof}
The sequence $(\mu_i)_{i\in\N}$ is monotonically nonincreasing since, for all $i \in \N$, $f(x_i) \le \mu_{i-1}$, hence $\mu_i=(1-p)\mu_{i-1}+ pf(x_i) \le \mu_{i-1}$. Therefore, for all $i \in \N$,
\begin{equation*}
f(x_i) \le \mu_{i-1} \le \mu_{-1} = f(x_0),
\end{equation*}
meaning that $(x_i)_{i\in\N}$ is contained in the sublevel set~\eqref{eq:InitialSublevelSet}.

Let us prove the second item of the proposition. Assume that $(x_i)_{i\in\N}$ has an accumulation point $x$. Let $(x_{i_k})_{k\in\N}$ be a subsequence converging to~$x$. Observe that
\begin{equation*}
\lim_{k \to \infty} f(x_{i_k}) = f(x),
\end{equation*}
since $f$ is differentiable, and in particular continuous, at $x$.
As $(\mu_i)_{i\in\N}$ is monotonically nonincreasing, it has a limit $\varphi \in \R \cup \{-\infty\}$. For all $k \in \N$,
\begin{equation*}
f(x_{i_k}) \le \mu_{i_k-1},
\end{equation*}
and letting $k$ tend to infinity yields
\begin{equation*}
f(x) = \lim_{k\to\infty} f(x_{i_k}) \le \lim_{k\to\infty} \mu_{i_k-1} = \varphi.
\end{equation*}
In particular, $\varphi$ is finite.

Let us show that $\varphi = \liminf_{i\to\infty} f(x_i)$. Let $(x_{j_k})_{k\in\N}$ be a subsequence such that
\begin{equation*}
\lim_{k \to \infty} f(x_{j_k}) = \liminf_{i\to\infty} f(x_i).
\end{equation*}
For all $k \in \N$,
\begin{align*}
\mu_{j_k} = (1-p) \mu_{j_k-1} + p f(x_{j_k}),
\end{align*}
and letting $k$ tend to infinity yields
\begin{equation*}
\varphi = (1-p) \varphi + p \liminf_{i\to\infty} f(x_i).
\end{equation*}
As $p>0$, this implies $\varphi = \liminf_{i\to\infty} f(x_i)$ (and, in particular, $\liminf_{i\to\infty} f(x_i) > -\infty$). To conclude, we observe that, for all $k \in \N$,
\begin{equation*}
f(x_k) \le \mu_{k-1},
\end{equation*}
hence
\begin{equation*}
\limsup_{k\to\infty} f(x_k) 
\leq \lim_{k\to\infty} \mu_{k-1} = \varphi = \liminf_{k\to\infty} f(x_k).
\end{equation*}
Therefore, $(f(x_k))_{k\in\N}$ converges to $\varphi$.
\end{proof}

\section{Convergence analysis for a continuous gradient}
\label{sec:ConvergenceAnalysisForContinuousGradient}
In this section, $\pgd$ (Algorithm~\ref{algo:PGD}) is analyzed under hypothesis~\ref{it:continuously_diff}. As mentioned after Remark~\ref{rem:FiniteOrInfinite}, only the nontrivial case where an infinite sequence is generated is considered here. Specifically, the first part of the second item of Theorem~\ref{thm:StationarityAccumulationPointsPGD}, restated in Theorem~\ref{thm:PGDaccumulatesAtBStationaryPoints} for convenience, is proven.

\begin{theorem}
\label{thm:PGDaccumulatesAtBStationaryPoints}
Let $(x_i)_{i \in \N}$ be a sequence generated by $\pgd$ (Algorithm~\ref{algo:PGD}), under~\ref{it:continuously_diff}. Then, all accumulation points of $(x_i)_{i \in \N}$ are B-stationary for~\eqref{eq:MinDiffFunctionClosedSet}.
If, moreover, $(x_i)_{i \in \N}$ has an isolated accumulation point, then $(x_i)_{i \in \N}$ converges.
\end{theorem}

The proof is divided into three parts. First, in section~\ref{subsec:PGDmapTerminantesBounded}, we show that, in a neighborhood of every point that is not B-stationary for~\eqref{eq:MinDiffFunctionClosedSet}, the $\pgd$ map (Algorithm~\ref{algo:PGDmap}) terminates after a bounded number of iterations. Then, in section~\ref{subsec:converges_to_x}, we prove that, if a subsequence $(x_{i_k})_{k \in \N}$ converges, then $(x_{i_k+1})_{k \in \N}$ also does, to the same limit. Finally, we combine the first two parts in section~\ref{subsec:proof_accumulates}: roughly, if $(x_{i_k})_{k \in \N}$ converges to $x$, then, from the second part,
\begin{equation*}
\norm{x_{i_k+1}-x_{i_k}} \to 0\quad
\text{when } k \to \infty,
\end{equation*}
but, from the first part, if $x$ is not B-stationary for~\eqref{eq:MinDiffFunctionClosedSet}, then the iterates of $\pgd$ move by at least a constant amount at each iteration. It is therefore impossible that $(x_{i_k})_{k \in \N}$ converges to a point that is not B-stationary for~\eqref{eq:MinDiffFunctionClosedSet}.

\subsection{First part: analysis of the PGD map}
\label{subsec:PGDmapTerminantesBounded}
In this section, we show that, if $\ushort{x} \in C$ is not B-stationary for~\eqref{eq:MinDiffFunctionClosedSet}, then the while loop in Algorithm~\ref{algo:PGDmap} terminates, in some neighborhood of $\ushort{x}$, for nonvanishing values of $\alpha$.
The intuition for this proof is that, for every $x$ close to $\ushort{x}$ and for every $y\in\proj{C}{x-\alpha\nabla f(x)}$,
\begin{equation*}
f(y) = f(x) + \ip{\nabla f(x)}{y-x} + \text{some remainder}.
\end{equation*}
The inner product $\ip{\nabla f(x)}{y-x}$ is negative, and larger in absolute value than some fraction of $\norm{\nabla f(x)}\norm{y-x}$ (Proposition~\ref{prop:ArbitrarySmallAlpha}). On the other hand, if $\alpha$ is small enough, the remainder (upper bounded in Proposition~\ref{prop:AcceptedStep}) is smaller than some arbitrarily small fraction of $\norm{\nabla f(x)}\norm{y-x}$. Therefore, for $\alpha$ small enough,
\begin{equation*}
f(y) < f(x) + c\ip{\nabla f(x)}{y-x}.
\end{equation*}

\begin{proposition}
\label{prop:ArbitrarySmallAlpha}
Assume that $f$ satisfies~\ref{it:continuously_diff}. Let $\ushort{x} \in C$ be non-B-stationary for~\eqref{eq:MinDiffFunctionClosedSet}, and $w \in \tancone{C}{\ushort{x}}$ be such that
\begin{equation}
\label{eq:def_w}
\ip{w}{-\nabla f(\ushort{x})} > 0.
\end{equation}
Let $\beta \in (0, 1)$. Define $\kappa \coloneq \sqrt{1 - \frac{\beta \ip{w}{-\nabla f(\ushort{x})}^2}{8 \norm{w}^2 \norm{\nabla f(\ushort{x})}^2}} \in (0, 1)$.
For every $\varepsilon \in (0, \infty)$, there exist $\alpha_{\ushort{x}} \in (0, \varepsilon]$ and $\oshort{\rho}(\alpha_{\ushort{x}}) \in (0, \infty)$ such that, for all $x \in \ball(\ushort{x}, \oshort{\rho}(\alpha_{\ushort{x}})) \cap C$ and $\alpha \in [\alpha_{\ushort{x}}, \alpha_{\ushort{x}}/\beta]$,
\begin{equation*}
\dist(x-\alpha\nabla f(x), C) \le \kappa \alpha \norm{\nabla f(x)},
\end{equation*}
which implies, for all $y \in\proj{C}{x-\alpha\nabla f(x)}$,
\begin{equation*}
\ip{\nabla f(x)}{y-x} \le - \sqrt{1-\kappa^2} \norm{\nabla f(x)} \norm{y-x}.
\end{equation*}
\end{proposition}

\begin{proof}
Let $\varepsilon \in (0, \infty)$ be fixed. We show that there exist $\alpha_{\ushort{x}} \in (0, \varepsilon]$ and $\oshort{\rho}(\alpha_{\ushort{x}}) \in (0, \infty)$ satisfying the required property.

Let $(w_i)_{i\in\N}$ be a sequence in $C$ converging to $\ushort{x}$, and $(t_i)_{i\in\N}$ be a sequence in $(0, \infty)$ such that
\begin{equation*}
\frac{w_i - \ushort{x}}{t_i} \xrightarrow{i\to\infty} w.
\end{equation*}
From the definition of $w$ in~\eqref{eq:def_w}, it holds for all $i \in \N$ large enough that
\begin{equation}
\label{eq:def_i_1}
\ip{w_i - \ushort{x}}{-\nabla f(\ushort{x})} > 0.
\end{equation}
As $\frac{1}{t_i}\frac{\norm{w_i-\ushort{x}}^2}{\ip{w_i-\ushort{x}}{-\nabla f(\ushort{x})}} \xrightarrow{i\to\infty} \frac{\norm{w}^2}{\ip{w}{-\nabla f(\ushort{x})}}$ and $t_i \xrightarrow{i\to\infty} 0$, it also holds for all $i \in \N$ large enough that
\begin{equation}
\label{eq:def_i_2}
\frac{\norm{w_i-\ushort{x}}^2}{\ip{w_i-\ushort{x}}{-\nabla f(\ushort{x})}} < \varepsilon.
\end{equation}
Similarly, it holds for all $i \in \N$ large enough that
\begin{equation}
\label{eq:def_i_3}
\frac{\ip{w_i-\ushort{x}}{-\nabla f(\ushort{x})}^2}{\norm{w_i-\ushort{x}}^2}
> \frac{\ip{w}{-\nabla f(\ushort{x})}^2}{2\norm{w}^2 }.
\end{equation}
Fix $i \in \N$ satisfying \eqref{eq:def_i_1}, \eqref{eq:def_i_2}, and \eqref{eq:def_i_3}.
Pick $\alpha_{\ushort{x}}$ such that
\begin{equation*}
\frac{\alpha_{\ushort{x}}}{2}
< \frac{\norm{w_i-\ushort{x}}^2}{\ip{w_i-\ushort{x}}{-\nabla f(\ushort{x})}}
< \alpha_{\ushort{x}}
< \varepsilon.
\end{equation*}
Since $\nabla f$ is continuous at $\ushort{x}$, there exists $\rho_0 \in (0, \infty)$ such that, for all $x \in \ball[\ushort{x}, \rho_0] \cap C$,
\begin{subequations}
\begin{gather}
\label{eq:def_rho0_1}
\ip{w_i-\ushort{x}}{-\nabla f(x)} > 0,\\
\label{eq:def_rho0_2}
\frac{\alpha_{\ushort{x}}}{2}
< \frac{\norm{w_i-\ushort{x}}^2}{\ip{w_i-\ushort{x}}{-\nabla f(x)}}
< \alpha_{\ushort{x}},\\
\label{eq:def_rho0_3}
\frac{\ip{w_i-\ushort{x}}{-\nabla f(x)}^2}{\norm{w_i-\ushort{x}}^2 \norm{\nabla f(x)}^2}
> \frac{\ip{w}{-\nabla f(\ushort{x})}^2}{2\norm{w}^2 \norm{\nabla f(\ushort{x})}^2}.
\end{gather}
\end{subequations}
We now establish the first inequality we have to prove: for an adequate value of $\oshort{\rho}(\alpha_{\ushort{x}})$, it holds for all $x \in \ball(\ushort{x}, \oshort{\rho}(\alpha_{\ushort{x}})) \cap C$ and $\alpha \in [\alpha_{\ushort{x}}, \alpha_{\ushort{x}}/\beta]$ that
\begin{equation*}
\norm{x - \alpha \nabla f(x) - y} \le \kappa \alpha \norm{\nabla f(x)}\quad
\forall y \in \proj{C}{x-\alpha\nabla f(x)},
\end{equation*}
which is equivalent to $\dist(x-\alpha\nabla f(x), C) \le \kappa \alpha \norm{\nabla f(x)}$.

For the moment, let $\oshort{\rho}(\alpha_{\ushort{x}}) \in (0, \rho_0]$ be arbitrary. For all $x \in \ball(\ushort{x}, \oshort{\rho}(\alpha_{\ushort{x}})) \cap C$, $\alpha \in [\alpha_{\ushort{x}}, \alpha_{\ushort{x}}/\beta]$, and $y \in \proj{C}{x-\alpha\nabla f(x)}$,
\begin{align*}
\norm{x - \alpha \nabla f(x) - y}^2
&\le \norm{x - \alpha \nabla f(x) - w_i}^2 \\
&= \norm{\ushort{x}-\alpha \nabla f(x) - w_i}^2 + 2 \ip{\ushort{x}-x}{\alpha \nabla f(x) + w_i-\ushort{x}} + \norm{\ushort{x}-x}^2 \\
&\le \norm{\ushort{x}-\alpha \nabla f(x) - w_i}^2 \\
&\qquad + 2 \oshort{\rho}(\alpha_{\ushort{x}}) \left(\alpha \norm{\nabla f(x)} + \norm{w_i-\ushort{x}}\right) + \oshort{\rho}(\alpha_{\ushort{x}})^2 \\
&\le \norm{\ushort{x}-\alpha \nabla f(x) - w_i}^2 \\
&\qquad + 2 \oshort{\rho}(\alpha_{\ushort{x}}) \left(\alpha \max_{z \in \ball[\ushort{x}, \rho_0] \cap C} \norm{\nabla f(z)} + \norm{w_i-\ushort{x}}\right) + \oshort{\rho}(\alpha_{\ushort{x}})^2 \\
&= \alpha^2 \norm{\nabla f(x)}^2 -2 \alpha \ip{w_i-\ushort{x}}{-\nabla f(x)} + \norm{w_i-\ushort{x}}^2 \\
&\qquad + 2 \oshort{\rho}(\alpha_{\ushort{x}}) \left(\alpha \max_{z \in \ball[\ushort{x}, \rho_0] \cap C} \norm{\nabla f(z)} + \norm{w_i-\ushort{x}}\right) + \oshort{\rho}(\alpha_{\ushort{x}})^2 \\
&\le \alpha^2 \norm{\nabla f(x)}^2 - \alpha \ip{w_i-\ushort{x}}{-\nabla f(x)} \\
&\qquad + 2 \oshort{\rho}(\alpha_{\ushort{x}}) \left(\frac{\alpha_{\ushort{x}}}{\beta} \max_{z \in \ball[\ushort{x}, \rho_0] \cap C} \norm{\nabla f(z)} + \norm{w_i-\ushort{x}}\right)
+ \oshort{\rho}(\alpha_{\ushort{x}})^2,
\end{align*}
where the last inequality follows from \eqref{eq:def_rho0_2} and the fact that $\alpha_{\ushort{x}} \le \alpha \le \frac{\alpha_{\ushort{x}}}{\beta}$.
Now, choose $\oshort{\rho}(\alpha_{\ushort{x}}) \in (0, \rho_0]$ small enough to ensure
\begin{align*}
&2 \oshort{\rho}(\alpha_{\ushort{x}}) \left(\frac{\alpha_{\ushort{x}}}{\beta} \max_{z \in \ball[\ushort{x}, \rho_0] \cap C} \norm{\nabla f(z)} + \norm{w_i-\ushort{x}}\right) + \oshort{\rho}(\alpha_{\ushort{x}})^2 \\
&\le \frac{\alpha_{\ushort{x}}}{2} \min_{z \in \ball[\ushort{x}, \rho_0] \cap C} \ip{w_i-\ushort{x}}{-\nabla f(z)}.
\end{align*}
Note that the right-hand side of this inequality is positive, from \eqref{eq:def_rho0_1}.
Combining this definition with the previous inequality, we arrive at
\begin{align*}
\norm{x - \alpha \nabla f(x) - y}^2
&\le \alpha^2 \norm{\nabla f(x)}^2 - \frac{\alpha}{2} \ip{w_i-\ushort{x}}{-\nabla f(x)}\\
&= \alpha^2 \norm{\nabla f(x)}^2 \left(1 - \frac{\ip{w_i-\ushort{x}}{-\nabla f(x)}}{2 \alpha \norm{\nabla f(x)}^2}\right)\\
&\le \alpha^2 \norm{\nabla f(x)}^2 \left(1 - \frac{\beta \ip{w_i-\ushort{x}}{-\nabla f(x)}}{2 \alpha_{\ushort{x}} \norm{\nabla f(x)}^2}\right)
\text{ as $\alpha \le \frac{\alpha_{\ushort{x}}}{\beta}$}\\
&\le \alpha^2 \norm{\nabla f(x)}^2 \left(1 - \frac{\beta \ip{w_i-\ushort{x}}{-\nabla f(x)}^2}{4 \norm{w_i-\ushort{x}}^2 \norm{\nabla f(x)}^2}\right)
\text{ from~\eqref{eq:def_rho0_2}}\\
&\le \alpha^2 \norm{\nabla f(x)}^2 \left(1 - \frac{\beta \ip{w}{-\nabla f(\ushort{x})}^2}{8 \norm{w}^2 \norm{\nabla f(\ushort{x})}^2}\right)
\text{ from~\eqref{eq:def_rho0_3}}\\
&= \kappa^2 \alpha^2 \norm{\nabla f(x)}^2.
\end{align*}
Thus, for all $x \in \ball(\ushort{x}, \oshort{\rho}(\alpha_{\ushort{x}})) \cap C$, $\alpha \in [\alpha_{\ushort{x}}, \alpha_{\ushort{x}}/\beta]$, and $y \in \proj{C}{x-\alpha\nabla f(x)}$,
\begin{equation*}
\norm{x - \alpha \nabla f(x) - y} \le \kappa \alpha \norm{\nabla f(x)}.
\end{equation*}
To conclude, we show that this inequality implies
\begin{equation}
\label{eq:ip_xy_nabla}
\ip{\frac{y-x}{\norm{y-x}}}{\frac{\nabla f(x)}{\norm{\nabla f(x)}}} \le - \sqrt{1-\kappa^2}.
\end{equation}
Indeed, defining $\theta \in \R$ such that $\ip{\frac{y-x}{\norm{y-x}}}{\frac{\nabla f(x)}{\norm{\nabla f(x)}}} = \cos(\theta)$, we have
\begin{equation*}
\norm{y-x}^2 + 2 \alpha \norm{\nabla f(x)} \norm{y-x} \cos(\theta) + \alpha^2 \norm{\nabla f(x)}^2
\le \alpha^2 \kappa^2 \norm{\nabla f(x)}^2.
\end{equation*}
This already shows that $\cos(\theta) < 0$. In addition, minimizing the left-hand side over all possible values of $\norm{y-x}$, we get
\begin{equation*}
- \alpha^2 \norm{\nabla f(x)}^2 \cos^2(\theta) + \alpha^2 \norm{\nabla f(x)}^2
\le \alpha^2 \kappa^2 \norm{\nabla f(x)}^2,
\end{equation*}
hence $\cos^2(\theta) \ge 1 - \kappa^2$, which establishes \eqref{eq:ip_xy_nabla}.
\end{proof}

\begin{proposition}
\label{prop:AcceptedStep}
Let $\ushort{\alpha} \in (0, \infty)$ and $\beta, c \in (0, 1)$.
Assume that $f$ satisfies~\ref{it:continuously_diff}. Let $\ushort{x} \in C$ be non-B-stationary for~\eqref{eq:MinDiffFunctionClosedSet}.
There exist $\alpha_{\ushort{x}} \in (0, \ushort{\alpha}]$ and $\rho \in (0, \infty)$ such that, for all $x \in \ball(\ushort{x}, \rho) \cap C$, $\alpha \in [\alpha_{\ushort{x}}, \alpha_{\ushort{x}}/\beta]$, and $y \in \proj{C}{x-\alpha\nabla f(x)}$,
\begin{equation*}
f(y) < f(x) + c \ip{\nabla f(x)}{y-x}.
\end{equation*}
\end{proposition}

\begin{proof}
Define $\kappa$ as in Proposition~\ref{prop:ArbitrarySmallAlpha}. Let $\delta \in (0, \infty)$ be small enough to ensure
\begin{subequations}
\begin{align}
\label{eq:def_alpha_x_small1}
\sup_{y \in \ball\left[\ushort{x}, \frac{7\delta}{2\beta}\norm{\nabla f(\ushort{x})}\right] \cap C \setminus \{\ushort{x}\}}
\frac{|f(y)-f(\ushort{x}) - \ip{\nabla f(\ushort{x})}{y-\ushort{x}}|}{\norm{y-\ushort{x}}}
&< \frac{(1-c)\sqrt{1-\kappa^2}\norm{\nabla f(\ushort{x})}}{4\left(1 + \frac{8}{3(1-\kappa)}\right)},\\
\label{eq:def_alpha_x_small2}
\sup_{y \in \ball\left[\ushort{x}, \frac{7\delta}{2\beta}\norm{\nabla f(\ushort{x})}\right] \cap C}
\norm{\nabla f(y)-\nabla f(\ushort{x})}
&< \frac{(1-c)\sqrt{1-\kappa^2}}{4} \norm{\nabla f(\ushort{x})}.
\end{align}
\end{subequations}
These inequalities are satisfied by all $\delta$ small enough, from the definition of the gradient for the first one, and because the gradient is continuous at $\ushort{x}$ for the second one.

Then, define $\varepsilon \coloneq \min\left\{\ushort{\alpha}, \delta\right\}$ and let $\alpha_{\ushort{x}} \in (0, \varepsilon]$ and $\oshort{\rho}(\alpha_{\ushort{x}})\in(0,\infty)$ be given by Proposition~\ref{prop:ArbitrarySmallAlpha}.
Define
\begin{equation*}
\rho \coloneq \min\left\{\oshort{\rho}(\alpha_{\ushort{x}}), \alpha_{\ushort{x}} \norm{\nabla f(\ushort{x})}\right\}.
\end{equation*}
Note that, for all $x \in \ball(\ushort{x}, \rho) \cap C$,
\begin{equation*}
\norm{x-\ushort{x}}
< \rho \le \alpha_{\ushort{x}} \norm{\nabla f(\ushort{x})}
< \frac{7\alpha_{\ushort{x}}}{2\beta} \norm{\nabla f(\ushort{x})}
\leq \frac{7\delta }{2\beta} \norm{\nabla f(\ushort{x})},
\end{equation*}
so that from~\eqref{eq:def_alpha_x_small2}, $\norm{\nabla f(x)-\nabla f(\ushort{x})} < \frac{\norm{\nabla f(\ushort{x})}}{4}$, which implies
\begin{align}
\nonumber
\frac{3}{4} \norm{\nabla f(\ushort{x})}
& < \norm{\nabla f(\ushort{x})}-\norm{\nabla f(x)-\nabla f(\ushort{x})} \\
& \le \norm{\nabla f(x)} \nonumber\\
& \le \norm{\nabla f(\ushort{x})}+\norm{\nabla f(x)-\nabla f(\ushort{x})} \nonumber\\
& < \frac{5}{4} \norm{\nabla f(\ushort{x})}.
\label{eq:nablas_ushortx_x}
\end{align}
For all $x \in \ball(\ushort{x}, \rho) \cap C$, $\alpha \in [\alpha_{\ushort{x}}, \alpha_{\ushort{x}}/\beta]$, and $y \in \proj{C}{x-\alpha\nabla f(x)}$,
\begin{align}
f(y)
&= f(x) + \ip{\nabla f(\ushort{x})}{y-x} \nonumber\\
&\qquad + \left(f(\ushort{x}) - f(x) - \ip{\nabla f(\ushort{x})}{\ushort{x}-x}\right) \nonumber\\
&\qquad + \left(f(y) - f(\ushort{x}) - \ip{\nabla f(\ushort{x})}{y-\ushort{x}}\right) \nonumber\\
&\le f(x) + \ip{\nabla f(\ushort{x})}{y-x} +
\frac{(1-c)\sqrt{1-\kappa^2}\norm{\nabla f(\ushort{x})}}{4\left(1 + \frac{8}{3(1-\kappa)}\right)} \left(\norm{\ushort{x}-x} + \norm{y-\ushort{x}}\right),
\label{eq:UpperBoundfDiffGrad}
\end{align}
where the inequality follows from \eqref{eq:def_alpha_x_small1}. Observe that
\begin{align*}
\norm{y-\ushort{x}}
&\le \norm{y-x} + \norm{x-\ushort{x}} \\
&\le 2 \alpha \norm{\nabla f(x)} + \rho
\text{ from \eqref{eq:ProjectedTranslationDistance}}\\
&\le \frac{2 \alpha_{\ushort{x}}}{\beta} \norm{\nabla f(x)}
+ \alpha_{\ushort{x}} \norm{\nabla f(\ushort{x})} \\
&< \frac{5\alpha_{\ushort{x}}}{2\beta} \norm{\nabla f(\ushort{x})}
+ \alpha_{\ushort{x}} \norm{\nabla f(\ushort{x})} 
\text{ from \eqref{eq:nablas_ushortx_x}} \\
&\le \frac{7\alpha_{\ushort{x}}}{2\beta} \norm{\nabla f(\ushort{x})}.
\end{align*}
We continue from \eqref{eq:UpperBoundfDiffGrad}:
\begin{align*}
f(y)
& \le f(x) + \ip{\nabla f(\ushort{x})}{y-x} +
\frac{(1-c)\sqrt{1-\kappa^2}\norm{\nabla f(\ushort{x})}}{
4\left(1 + \frac{8}{3(1-\kappa)}\right)} \left(2\norm{\ushort{x}-x} + \norm{y-x}\right) \\
&\overset{\text{(a)}}{<} f(x) + \ip{\nabla f(\ushort{x})}{y-x}
+ \frac{(1-c)\sqrt{1-\kappa^2}\norm{\nabla f(\ushort{x})}}{4} \norm{y-x} \\
&\le f(x) + \ip{\nabla f(x)}{y-x} + \norm{\nabla f(\ushort{x})-\nabla f(x)} \norm{y-x} \\
&\qquad + \frac{(1-c)\sqrt{1-\kappa^2}\norm{\nabla f(\ushort{x})}}{4} \norm{y-x} \\
&\le f(x) + \ip{\nabla f(x)}{y-x} + \frac{(1-c)\sqrt{1-\kappa^2}\norm{\nabla f(\ushort{x})}}{2} \norm{y-x}
\text{ from~\eqref{eq:def_alpha_x_small2}}\\
&< f(x) + \ip{\nabla f(x)}{y-x} + (1-c)\sqrt{1-\kappa^2} \norm{\nabla f(x)}\norm{y-x}
\text{ from~\eqref{eq:nablas_ushortx_x}}\\
&\le f(x) + \ip{\nabla f(x)}{y-x} - (1-c) \ip{\nabla f(x)}{y-x}
\text{ from Proposition~\ref{prop:ArbitrarySmallAlpha}}\\
&= f(x) + c \ip{\nabla f(x)}{y-x}.
\end{align*}
Inequality (a) is true because
\begin{align*}
\norm{y-x}
&\ge \alpha \norm{\nabla f(x)} - \norm{x-\alpha \nabla f(x)-y}
\text{ by the triangle inequality}\\
&= \alpha \norm{\nabla f(x)} - \dist(x-\alpha\nabla f(x), C) \\
&\ge (1-\kappa) \alpha \norm{\nabla f(x)}
\text{ from Proposition~\ref{prop:ArbitrarySmallAlpha}}\\
&\ge \frac{3}{4} (1-\kappa) \rho
\text{ from~\eqref{eq:nablas_ushortx_x}, the definition of $\rho$, and $\alpha_{\ushort{x}} \le \alpha$}\\
&> \frac{3}{4} (1-\kappa) \norm{x-\ushort{x}}.
\end{align*}
\end{proof}

\subsection{Second part: convergence of successive iterates}
\label{subsec:converges_to_x}

\begin{proposition}
\label{prop:convergence_ik}
Let $(x_i)_{i \in \N}$ be a sequence generated by $\pgd$ (Algorithm~\ref{algo:PGD}), under~\ref{it:continuously_diff}, and $x$ be an accumulation point.
Then, for every subsequence $(x_{i_k})_{k\in\N}$ converging to $x$, the sequence $(x_{i_k+1})_{k\in\N}$ also converges to $x$.
\end{proposition}

\begin{proof}
Let $(x_{i_k})_{k\in\N}$ be a subsequence converging to $x$. We show that $(x_{i_k+1})_{k\in\N}$ also converges to $x$.

If the nonmonotonicity rule is set to ``average'', this is a direct consequence of Proposition~\ref{prop:NonmonotonePGDAverage}. Indeed, for all $i \in \N$, from~\eqref{eq:KanzowMehlitzDecay},
\begin{equation*}
f(x_{i+1}) \le \mu_i - \frac{c}{2 \oshort{\alpha}} \norm{x_{i+1}-x_i}^2 \le \mu_i.    
\end{equation*}
From Proposition~\ref{prop:NonmonotonePGDAverage}, $(f(x_{i+1}))_{i\in\N}$ and $(\mu_i)_{i\in\N}$ converge to the same limit. Therefore,
\begin{equation*}
 \left(\mu_i - \frac{c}{2 \oshort{\alpha}} \norm{x_{i+1}-x_i}^2\right)_{i\in\N}
\end{equation*}
also converges to this limit.
This implies that $(\norm{x_{i+1}-x_i})_{i\in\N}$ converges to $0$, hence $(\norm{x_{i_k+1}-x_{i_k}})_{k\in\N}$ converges to $0$, and $(x_{i_k+1})_{k\in\N}$ converges to the same limit as $(x_{i_k})_{k\in\N}$, that is, $x$.

Now, let us consider the ``max'' rule case. It suffices to show that $x$ is an accumulation point of every subsequence of $(x_{i_k+1})_{k\in\N}$. Thus, we show the following: for every subsequence $(i_{j_k})_{k\in\N}$ of $(i_k)_{k\in\N}$, there exists a subsequence of $(x_{i_{j_k}+1})_{k\in\N}$ that converges to $x$.
Let $(i_{j_k})_{k\in\N}$ be a subsequence of $(i_k)_{k\in\N}$.
For all $i \in \N$, let $g(i) \in \argmax_{j \in \{\max\{0, i-l\}, \dots, i\}} f(x_j)$, as in Proposition~\ref{prop:NonmonotonePGD}. By the third statement of Proposition~\ref{prop:NonmonotonePGD}, the sequence $(f(x_{g(i)}))_{i\in\N}$ converges to $\varphi \in [f(x), f(x_0)]$.
For every $k \in \N$, letting $\alpha_{i_{j_k}} \in (0, \oshort{\alpha}]$ be the number such that $x_{i_{j_k}+1} \in \proj{C}{x_{i_{j_k}} - \alpha_{i_{j_k}} \nabla f(x_{i_{j_k}})}$, by \eqref{eq:ProjectedTranslationDistance},
\begin{equation*}
\norm{x_{i_{j_k}+1}-x_{i_{j_k}}}
\le 2 \alpha_{i_{j_k}} \norm{\nabla f(x_{i_{j_k}})}
\le 2 \oshort{\alpha} \norm{\nabla f(x_{i_{j_k}})}.
\end{equation*}
Thus, since $(x_{i_{j_k}})_{k\in\N}$ is bounded and $\nabla f$ is locally bounded (as it is continuous), the sequence $(x_{i_{j_k}+1})_{k\in\N}$ is bounded. If we replace $(i_{j_k})_{k\in\N}$ with a subsequence, we can assume that $(x_{i_{j_k}+1})_{k\in\N}$ converges.

Iterating the reasoning, we can assume that $(x_{i_{j_k}+s})_{k\in\N}$ converges to some $x^{s}\in C$ for every $s \in \{0, \dots, l+1\}$. By definition of $x$, $x^{0}=x$.

Observe that, from the continuity of $f$,
\begin{align*}
f(x_{g(i_{j_k}+l+1)})
&= \max\{f(x_{i_{j_k}+1}), \dots, f(x_{i_{j_k}+l+1})\} \\
&\to \max\{f(x^{1}), \dots, f(x^{l+1})\}
\text{ when } k \to \infty.
\end{align*}
In particular, there exists $s_1 \in \{1, \dots, l+1\}$ such that
\begin{equation}
\label{eq:def_s_1}
f(x^{s_1}) = \varphi.
\end{equation}
Let $s_1$ be the smallest such integer.
For all $k \in \N$, from the condition in line~\ref{algo:PGDmap:NonmonotoneArmijoCondition} of Algorithm~\ref{algo:PGDmap} and~\eqref{eq:KanzowMehlitzDecay},
\begin{align*}
f(x_{i_{j_k}+s_1})
&\le f(x_{g(i_{j_k}+s_1-1)}) - \frac{c}{2\oshort{\alpha}} \norm{x_{i_{j_k}+s_1}-x_{i_{j_k}+s_1-1}}^2,
\end{align*}
and letting $k$ tend to infinity yields
\begin{equation*}
\varphi
= f(x^{s_1})
\le \varphi - \frac{c}{2\oshort{\alpha}} \norm{x^{s_1}-x^{s_1-1}}^2.
\end{equation*}
Consequently, $x^{s_1} = x^{s_1-1}$. In particular, $f(x^{s_1-1}) =f(x^{s_1}) = \varphi$. Therefore, $s_1=1$, otherwise it would not be the smallest integer satisfying \eqref{eq:def_s_1}. The equality $x^{s_1} = x^{s_1-1}$ then reduces to $x^{1}=x^{0}=x$ and, when $k \to \infty$,
\begin{equation*}
x_{i_{j_k}+1} \to x^{1} = x.
\end{equation*}
\end{proof}

\subsection{Third part: proof of Theorem~\ref{thm:PGDaccumulatesAtBStationaryPoints}}
\label{subsec:proof_accumulates}
Let $\ushort{x}$ be an accumulation point of $(x_i)_{i\in\N}$. Assume, for the sake of contradiction, that $\ushort{x}$ is not B-stationary for~\eqref{eq:MinDiffFunctionClosedSet}. Let $(x_{i_k})_{k\in\N}$ be a subsequence converging to $\ushort{x}$.

Let $\alpha_{\ushort{x}}$ and $\rho$ be as in Proposition~\ref{prop:AcceptedStep}. For all $k \in \N$ large enough, $x_{i_k} \! \in \ball(\ushort{x}, \rho) \cap C$. Thus, when Algorithm~\ref{algo:PGDmap} is called at point $x_{i_k}$, the condition in line~\ref{algo:PGDmap:NonmonotoneArmijoCondition} stops being fulfilled for some $\alpha_{i_k} \ge \alpha_{\ushort{x}}$, meaning that
\begin{equation*}
x_{i_k+1} \in \proj{C}{x_{i_k}-\alpha_{i_k} \nabla f(x_{i_k})}
\text{ for some } \alpha_{i_k} \in [\alpha_{\ushort{x}}, \oshort{\alpha}].
\end{equation*}
Replacing $(i_k)_{k\in\N}$ with a subsequence if necessary, we can assume that $(\alpha_{i_k})_{k\in\N}$ converges to some $\alpha_{\lim} \in [\alpha_{\ushort{x}}, \oshort{\alpha}]$.

For all $k \in \N$, we have
\begin{equation*}
\norm{x_{i_k} - \alpha_{i_k} \nabla f(x_{i_k}) - x_{i_k+1}}
= \dist(x_{i_k} - \alpha_{i_k} \nabla f(x_{i_k}), C),
\end{equation*}
and since the distance to a nonempty closed set is a continuous function, we can take this equality to the limit. We use the fact that $x_{i_k+1} \to \ushort{x}$ when $k \to \infty$, from Proposition~\ref{prop:convergence_ik}. This yields
\begin{equation*}
\norm{\alpha_{\lim} \nabla f(\ushort{x})}
= \dist(\ushort{x} - \alpha_{\lim}\nabla f(\ushort{x}), C),
\end{equation*}
which means that $\ushort{x} \in \proj{C}{\ushort{x} - \alpha_{\lim}\nabla f(\ushort{x})}$. In particular, $-\nabla f(\ushort{x}) \in \proxnorcone{C}{\ushort{x}} \subseteq \regnorcone{C}{\ushort{x}}$, which contradicts our assumption that $\ushort{x}$ is not B-stationary for~\eqref{eq:MinDiffFunctionClosedSet}. We have therefore proven that every accumulation point is B-stationary.

Finally, if $(x_i)_{i\in\N}$ has an isolated accumulation point, then the sequence $(x_i)_{i\in\N}$ converges, from Proposition~\ref{prop:convergence_ik} and \cite[Lemma~4.10]{MoreSorensen}.

\section{Convergence analysis for a locally Lipschitz continuous gradient}
\label{sec:ConvergenceAnalysisForLocallyLipschitzContinuousGradient}
In this section, $\pgd$ (Algorithm~\ref{algo:PGD}) is analyzed under hypothesis~\ref{it:Lipschitz_diff}. As mentioned after Remark~\ref{rem:FiniteOrInfinite}, only the nontrivial case where an infinite sequence is generated is considered here. Specifically, the second part of the second item of Theorem~\ref{thm:StationarityAccumulationPointsPGD}, restated in Theorem~\ref{thm:PGDaccumulatesAtPStationaryPoints} for convenience, is proven based on Proposition~\ref{prop:ArmijoConditionLipschitzContinuousGradient} and Corollary~\ref{coro:PGDmapTerminationLipschitzContinuousGradient}, which state that, for every $\ushort{x} \in C$ and every input $x$ sufficiently close to $\ushort{x}$, the $\pgd$ map (Algorithm~\ref{algo:PGDmap}) terminates after at most a given number of iterations that depends only on~$\ushort{x}$.

\begin{proposition}
\label{prop:ArmijoConditionLipschitzContinuousGradient}
Assume that $f$ satisfies~\ref{it:Lipschitz_diff}.
Let $\ushort{x} \in C$, $\oshort{\alpha} \in (0, \infty)$, $c \in (0, 1)$, and $\rho \in (0, \infty)$. Let $\oshort{\rho} \in \left[\rho + 2 \oshort{\alpha} \max_{x \in \ball[\ushort{x}, \rho] \cap C} \norm{\nabla f(x)}, \infty\right)$ and define $\alpha_* \coloneq (1-c)/\lip_{\ball[\ushort{x}, \oshort{\rho}]}(\nabla f)$. Then, for all $x \in \ball[\ushort{x}, \rho] \cap C$, $\alpha \in [0, \min\{\alpha_*, \oshort{\alpha}\}]$, and $y \in \proj{C}{x-\alpha\nabla f(x)}$,
\begin{equation*}
f(y) \le f(x) + c \ip{\nabla f(x)}{y-x}.
\end{equation*}
\end{proposition}

\begin{proof}
For all $x \in \ball[\ushort{x}, \rho] \cap C$ and $\alpha \in [0, \oshort{\alpha}]$, $\proj{C}{x-\alpha\nabla f(x)} \subseteq \ball[\ushort{x}, \oshort{\rho}]$; indeed, for all $y \in \proj{C}{x-\alpha\nabla f(x)}$,
\begin{equation*}
\norm{y-\ushort{x}}
\le \norm{y-x} + \norm{x-\ushort{x}}
\le 2 \alpha \norm{\nabla f(x)} + \rho
\le \oshort{\rho},
\end{equation*}
where the second inequality follows from~\eqref{eq:ProjectedTranslationDistance}. Thus, by~\eqref{eq:InequalityLipschitzContinuousGradient} and~\eqref{eq:ProjectedTranslationIP}, for all $x \in \ball[\ushort{x}, \rho] \cap C$, $\alpha \in [0, \min\{\alpha_*, \oshort{\alpha}\}]$, and $y \in \proj{C}{x-\alpha\nabla f(x)}$,
\begin{align*}
f(y)
&\le f(x) + \ip{\nabla f(x)}{y-x} + \frac{1}{2} \lip_{\ball[\ushort{x}, \oshort{\rho}]}(\nabla f) \norm{y-x}^2\\
&\le f(x) + \left(1-\alpha\lip_{\ball[\ushort{x}, \oshort{\rho}]}(\nabla f)\right) \ip{\nabla f(x)}{y-x}\\
&\le f(x) + c \ip{\nabla f(x)}{y-x}.
\end{align*}
\end{proof}

\begin{corollary}
\label{coro:PGDmapTerminationLipschitzContinuousGradient}
Consider Algorithm~\ref{algo:PGDmap} under hypothesis~\ref{it:Lipschitz_diff}. Given $\ushort{x} \in C$ and $\rho \in (0, \infty)$, let $\oshort{\rho}$ and $\alpha_*$ be as in Proposition~\ref{prop:ArmijoConditionLipschitzContinuousGradient}.
Then, for every $x \in \ball[\ushort{x}, \rho] \cap C$, the while loop terminates with a step size $\alpha \in [\min\{\ushort{\alpha}, \beta\alpha_*\}, \oshort{\alpha}]$ and hence after at most $\max\{0, \lceil\ln(\alpha_*/\alpha_0)/\ln(\beta)\rceil\}$ iterations, where $\alpha_0$ is the step size chosen in line~\ref{algo:PGDmap:InitialStepSize}.
\end{corollary}

\begin{proof}
At the latest, the while loop ends after iteration $i \in \N \setminus \{0\}$ with $\alpha = \alpha_0\beta^i$ such that $\alpha/\beta > \alpha_*$. In that case, $i < 1+\ln(\alpha_*/\alpha_0)/\ln(\beta)$ and thus $i \le \lceil\ln(\alpha_*/\alpha_0)/\ln(\beta)\rceil$.
\end{proof}

\begin{theorem}
\label{thm:PGDaccumulatesAtPStationaryPoints}
Assume that $f$ satisfies~\ref{it:Lipschitz_diff}. Let $(x_i)_{i \in \N}$ be a sequence generated by $\pgd$ (Algorithm~\ref{algo:PGD}). Then, all accumulation points of $(x_i)_{i \in \N}$ are P-stationary for~\eqref{eq:MinDiffFunctionClosedSet}. Moreover, for every convergent subsequence $(x_{i_j})_{j \in \N}$,
\begin{equation}
\label{eq:PGDsubsequencePstationarityMeasure}
\lim_{j \to \infty} \dist(-\nabla f(x_{i_j+1}), \proxnorcone{C}{x_{i_j+1}}) = 0.
\end{equation}
\end{theorem}

\begin{proof}
Assume that a subsequence $(x_{i_j})_{j \in \N}$ converges to $\ushort{x} \in C$. Given $\rho \in (0, \infty)$, let $\oshort{\rho}$ and $\alpha_*$ be as in Proposition~\ref{prop:ArmijoConditionLipschitzContinuousGradient}. Define
\begin{equation*}
I \coloneq \left[\min\left\{\ushort{\alpha}, \beta\alpha_*\right\}, \oshort{\alpha}\right].
\end{equation*}
There exists $j_* \in \N$ such that, for all integers $j \ge j_*$, $x_{i_j} \in \ball[\ushort{x}, \rho]$, thus, by Corollary~\ref{coro:PGDmapTerminationLipschitzContinuousGradient}, $x_{i_j+1} \in \proj{C}{x_{i_j}-\alpha_{i_j}\nabla f(x_{i_j})}$ with $\alpha_{i_j} \in I$, and hence
\begin{equation*}
\norm{x_{i_j+1}-(x_{i_j}-\alpha_{i_j}\nabla f(x_{i_j}))} = \dist(x_{i_j}-\alpha_{i_j}\nabla f(x_{i_j}), C).
\end{equation*}
Since $I$ is compact, a subsequence $(\alpha_{i_{j_k}})_{k \in \N}$ converges to $\alpha \in I$. Moreover, there exists $k_* \in \N$ such that $j_{k_*} \ge j_*$. Furthermore, by Proposition~\ref{prop:convergence_ik}, $(x_{i_j+1})_{j \in \N}$ converges to $\ushort{x}$. Therefore, for all integers $k \ge k_*$,
\begin{equation*}
\norm{x_{i_{j_k}+1}-(x_{i_{j_k}}-\alpha_{i_{j_k}}\nabla f(x_{i_{j_k}}))} = \dist(x_{i_{j_k}}-\alpha_{i_{j_k}}\nabla f(x_{i_{j_k}}), C),
\end{equation*}
and letting $k$ tend to infinity yields
\begin{equation*}
\norm{\ushort{x}-(\ushort{x}-\alpha\nabla f(\ushort{x}))} = \dist(\ushort{x}-\alpha\nabla f(\ushort{x}), C).
\end{equation*}
It follows that $\ushort{x} \in \proj{C}{\ushort{x}-\alpha\nabla f(\ushort{x})}$, which implies that $-\nabla f(\ushort{x}) \in \proxnorcone{C}{\ushort{x}}$.

We now establish \eqref{eq:PGDsubsequencePstationarityMeasure}. Recall that, for all integers $j \ge j_*$, $x_{i_j+1} \in \proj{C}{x_{i_j}-\alpha_{i_j}\nabla f(x_{i_j})}$ with $\alpha_{i_j} \in I$, hence $\frac{1}{\alpha_{i_j}}(x_{i_j}-x_{i_j+1})-\nabla f(x_{i_j}) \in \proxnorcone{C}{x_{i_j+1}}$, and thus
\begin{align*}
\dist(-\nabla f(x_{i_j+1}), \proxnorcone{C}{x_{i_j+1}})
&\le \norm{-\nabla f(x_{i_j+1})-(\frac{1}{\alpha_{i_j}}(x_{i_j}-x_{i_j+1})-\nabla f(x_{i_j}))}\\
&\le \frac{1}{\alpha_{i_j}} \norm{x_{i_j+1}-x_{i_j}} + \norm{\nabla f(x_{i_j+1})-\nabla f(x_{i_j})}\\
&\to 0 \text{ when } j \to \infty,
\end{align*}
since $(x_{i_j})_{j \in \N}$ and $(x_{i_j+1})_{j \in \N}$ converge to $\ushort{x}$ and $(\alpha_{i_j})_{j \in \N}$ is bounded away from zero.
\end{proof}

Proposition~\ref{prop:PGDboundedSequence} considers the case where $\pgd$ generates a bounded sequence.

\begin{proposition}
\label{prop:PGDboundedSequence}
Assume that $f$ satisfies~\ref{it:Lipschitz_diff}. Let $(x_i)_{i \in \N}$ be a sequence generated by $\pgd$ (Algorithm~\ref{algo:PGD}). If $(x_i)_{i \in \N}$ is bounded, which is the case if the sublevel set~\eqref{eq:InitialSublevelSet} is bounded, then all of its accumulation points, of which there exists at least one, are P-stationary for~\eqref{eq:MinDiffFunctionClosedSet} and have the same image by $f$, and
\begin{equation}
\label{eq:PGDboundedSequencePstationarityMeasure}
\lim_{i \to \infty} \dist(-\nabla f(x_i), \proxnorcone{C}{x_i}) = 0.
\end{equation}
\end{proposition}

\begin{proof}
Assume that $(x_i)_{i \in \N}$ is bounded. It suffices to establish \eqref{eq:PGDboundedSequencePstationarityMeasure} and to prove that all accumulation points of $(x_i)_{i \in \N}$ have the same image by $f$; the other statements follow from Theorem~\ref{thm:PGDaccumulatesAtPStationaryPoints}.

The proof that all accumulation points of $(x_i)_{i \in \N}$ have the same image by $f$ is based on the argument given in the proof of \cite[Theorem~65]{Polak1971}. Assume that $(x_{i_k})_{k \in \N}$ and $(x_{j_k})_{k \in \N}$ converge respectively to $\ushort{x}$ and $\oshort{x}$. Being bounded, the sequence $(x_i)_{i \in \N}$ is contained in a compact set. By Propositions~\ref{prop:NonmonotonePGD} and~\ref{prop:NonmonotonePGDAverage}, the sequence $(f(x_i))_{i \in \N}$ converges; Proposition~\ref{prop:NonmonotonePGD} applies because a continuous real-valued function is bounded from below and uniformly continuous on every compact set \cite[Propositions~1.3.3 and 1.3.5]{Willem}. Therefore, $f(\ushort{x}) = \lim_{k \to \infty} f(x_{i_k}) = \lim_{i \to \infty} f(x_i) = \lim_{k \to \infty} f(x_{j_k}) = f(\oshort{x})$.

Let us establish \eqref{eq:PGDboundedSequencePstationarityMeasure}. Assume, for the sake of contradiction, that \eqref{eq:PGDboundedSequencePstationarityMeasure} does not hold. Then, there exist $\varepsilon \in (0, \infty)$ and a subsequence $(x_{i_j})_{j \in \N}$ such that $i_0 \ge 1$ and $\dist(-\nabla f(x_{i_j}), \proxnorcone{C}{x_{i_j}}) > \varepsilon$ for all $j \in \N$. Since $(x_{i_j-1})_{j \in \N}$ is bounded, it contains a subsequence $(x_{i_{j_k}-1})_{k \in \N}$ that converges to a point $\ushort{x} \in C$. Therefore, by \eqref{eq:PGDsubsequencePstationarityMeasure},
\begin{align*}
\lim_{k \to \infty} \dist(-\nabla f(x_{i_{j_k}}), \proxnorcone{C}{x_{i_{j_k}}}) = 0,
\end{align*}
a contradiction.
\end{proof}

\section{Examples of feasible sets on which PGD can be practically implemented}
\label{sec:ExamplesFeasibleSets}
Examples of a set $C$ on which $\pgd$ can be practically implemented include:
\begin{enumerate}
\item the closed cone $\sparse{n}{s}$ of $s$-sparse vectors of $\R^n$, namely those having at most $s$ nonzero components, $n$ and $s$ being positive integers such that $s < n$;
\item the closed cone $\sparse{n}{s} \cap \R_+^n$ of nonnegative $s$-sparse vectors of $\R^n$;
\item the determinantal variety \cite[Lecture~9]{Harris}
\begin{equation*}
\R_{\le r}^{m \times n} \coloneq \{X \in \R^{m \times n} \mid \rank X \le r\},
\end{equation*}
$m$, $n$, and $r$ being positive integers such that $r < \min\{m,n\}$;
\item the closed cone
\begin{equation*}
\mathrm{S}_{\le r}^+(n) \coloneq \{X \in \R_{\le r}^{n \times n} \mid X^\tp = X,\, X \succeq 0\}
\end{equation*}
of order-$n$ real symmetric positive-semidefinite matrices of rank at most $r$, $n$ and $r$ being positive integers such that $r < n$.
\end{enumerate}
Indeed, for every set in this list, the projection map, the tangent cone, the regular normal cone, and the normal cone are explicitly known; see \cite[\S\S 6 and 7.4]{OlikierGallivanAbsil2023} and references therein. In particular, it is known that these sets are not Clarke regular at infinitely many points. In this section, we prove that, for these sets, regular normals are proximal normals.

As detailed in \cite{OlikierGallivanAbsil2023}, if $C$ is a set in this list, then there exist a positive integer $p$ and disjoint nonempty smooth submanifolds $S_0, \dots, S_p$ of $\mathcal{E}$ such that $\overline{S_p} = C$ and, for all $i \in \{0, \dots, p\}$, $\overline{S_i} = \bigcup_{j=0}^i S_j$. This implies that $\{S_0, \dots, S_p\}$ is a \emph{stratification} of $C$ satisfying the \emph{condition of the frontier} \cite[\S 5]{Mather}. Thus, $C$ is called a \emph{stratified set} and $S_0, \dots, S_p$ are called the \emph{strata} of $\{S_0, \dots, S_p\}$.

\begin{proposition}
\label{prop:ProximalRegularLowSparsityOrRank}
Let $C$ be a set in the list. For all $x \in C$,
\begin{equation*}
\proxnorcone{C}{x} = \regnorcone{C}{x}
\end{equation*}
and, if $x \notin S_p$, then
\begin{equation*}
\regnorcone{C}{x} \subsetneq \norcone{C}{x}.
\end{equation*}
\end{proposition}

Since the proof of Proposition~\ref{prop:ProximalRegularLowSparsityOrRank} relies on significantly different concepts than those previously used, we present it in Appendix~\ref{app:examples_proof}.

\section{Comparison of PGD and P$^\mathbf{2}$GD on a simple example}
\label{sec:PGDvsP2GD}
$\ppgd$, which is short for projected-projected gradient descent, was introduced in \cite[Algorithm~3]{SchneiderUschmajew2015} for $C \coloneq \R_{\le r}^{m \times n}$ and extended to an arbitrary set $C$ in \cite[Algorithm~5.1]{OlikierGallivanAbsil2023}. It works like $\pgd$ except that it involves an additional projection: given $x \in C$ as input, the $\ppgd$ map \cite[Algorithm~5.1]{OlikierGallivanAbsil2023} performs a backtracking projected line search along an arbitrary projection $g$ of $-\nabla f(x)$ onto $\tancone{C}{x}$: it computes an arbitrary projection $y$ of $x+\alpha g$ onto $C$ for decreasing values of the step size $\alpha \in (0, \infty)$ until $y$ satisfies an Armijo condition.

As pointed out in \cite[\S 3.2]{Scholtes2004}, the convergence of optimization algorithms that use descent directions in the tangent cone, such as $\ppgd$, often suffers from the discontinuity of the tangent cone. In \cite[\S 2.2]{LevinKileelBoumal2023}, on an instance of~\eqref{eq:MinDiffFunctionClosedSet} where $\mathcal{E} \coloneq \R^{3 \times 3}$ and $C \coloneq \R_{\le 2}^{3 \times 3}$, $\ppgd$ is proven to generate a sequence converging to a point of rank one that is M-stationary but not B-stationary. Several methods are compared numerically on this instance in \cite[\S 8.2]{OlikierGallivanAbsil2024}.

In this section, monotone $\pgd$ and $\ppgd$ are compared analytically on the instance of~\eqref{eq:MinDiffFunctionClosedSet} where $\mathcal{E} \coloneq \R^2$, $C \coloneq \sparse{2}{1}$, $f(x) \coloneq \frac{1}{2} \norm{x-x_*}^2$ for all $x \in \R^2$, $x_* \coloneq (a, 0)$, and $a \in \R \setminus \{0\}$. For all $x \in \R^2$, $\nabla f(x) = x-x_*$. Thus, the global Lipschitz constant of $\nabla f$ is $1$; in particular, $f$ satisfies~\ref{it:Lipschitz_diff}. Both algorithms are used with $\ushort{\alpha} \coloneq \oshort{\alpha} \coloneq \alpha \in (0, 2)$ and an arbitrary $\beta \in (0, 1)$. The initial iterate is $(0, b)$ for some $b \in \R \setminus \{0\}$.

We recall from \cite[Proposition~7.13]{OlikierGallivanAbsil2023} that $\tancone{\sparse{2}{1}}{0, 0} = \sparse{2}{1}$ and, for all $t \in \R \setminus \{0\}$,
\begin{align*}
\tancone{\sparse{2}{1}}{0, t} = \{0\} \times \R,&&
\tancone{\sparse{2}{1}}{t, 0} = \R \times \{0\},
\end{align*}
from \cite[Propositions~7.16 and 7.17]{OlikierGallivanAbsil2023} that
\begin{equation*}
\regnorcone{\sparse{2}{1}}{0, 0} = \{(0, 0)\} \subsetneq \sparse{2}{1} = \norcone{\sparse{2}{1}}{0, 0}
\end{equation*}
and, for all $t \in \R \setminus \{0\}$,
\begin{align*}
\regnorcone{\sparse{2}{1}}{0, t} = \R \times \{0\},&&
\regnorcone{\sparse{2}{1}}{t, 0} = \{0\} \times \R,
\end{align*}
and from Proposition~\ref{prop:ProximalRegularLowSparsityOrRank} that $\proxnorcone{\sparse{2}{1}}{x} = \regnorcone{\sparse{2}{1}}{x}$ for all $x \in \sparse{2}{1}$.

Proposition~\ref{prop:PGDvsP2GD} explicitly describes the sequences generated by $\pgd$ and $\ppgd$ for $c$ small enough to ensure that, at every iteration, the monotone Armijo condition is satisfied with the initial step size and thus no backtracking occurs. We omit its proof, which consists in elementary computations.

\begin{proposition}
\label{prop:PGDvsP2GD}
If $\alpha = 1$ and $c \in (0, \frac{1}{2}]$, then $\pgd$ and $\ppgd$ generate the finite sequences $((0, b), (a, 0))$ and $((0, b), (0, 0), (a, 0))$, respectively.
If $\alpha \ne 1$, then both algorithms generate infinite sequences.
\begin{itemize}
\item For every $c \in (0, \frac{2-\alpha}{2}]$, $\ppgd$ generates the sequence $((0, (1-\alpha)^ib))_{i \in \N}$, which converges to $(0, 0)$.
\item For every $c \in (0, \frac{2-\alpha}{4})$:
\begin{itemize}
\item if $\alpha|a|/|b| > |1-\alpha|$, then $\pgd$ generates the sequence
\begin{equation*}
((0, b), ((a(1-(1-\alpha)^{i+1}), 0))_{i \in \N});
\end{equation*}
\item if $\alpha|a|/|b| \le |1-\alpha|$, then $i_* \coloneq \left\lfloor\frac{\ln(\alpha|a|/|b|)}{\ln(|1-\alpha|)}\right\rfloor \in \N \setminus \{0\}$ and $\pgd$ generates the sequence
\begin{equation*}
(((0, (1-\alpha)^ib))_{i=0}^{i=i_*}, ((a(1-(1-\alpha)^{i+1}), 0))_{i \in \N})
\end{equation*}
if $\frac{\ln(\alpha|a|/|b|)}{\ln(|1-\alpha|)} \notin \N$ and
\begin{align*}
& (((0, (1-\alpha)^ib))_{i=0}^{i=i_*}, ((a(1-(1-\alpha)^{i+1}), 0))_{i \in \N})\\
\text{or } & (((0, (1-\alpha)^ib))_{i=0}^{i=i_*+1}, ((a(1-(1-\alpha)^{i+1}), 0))_{i \in \N})
\end{align*}
if $\frac{\ln(\alpha|a|/|b|)}{\ln(|1-\alpha|)} \in \N$.
\end{itemize}
Thus, every sequence generated by $\pgd$ converges to $(a, 0)$.
\end{itemize}
\end{proposition}

In conclusion, if $\alpha \ne 1$, then $\ppgd$ converges to $(0, 0)$, which is M-stationary but not B-stationary, while $\pgd$ converges to $(a, 0)$, which is P-stationary and even a global minimizer of $f|_{\sparse{2}{1}}$ (and $f$). This is illustrated in Figure~\ref{fig:PGDvsP2GD} for some choice of $a$, $b$, and $\alpha$.

By Proposition~\ref{prop:PGDvsP2GD}, for every sequence $(x_i)_{i \in \N}$ generated by $\ppgd$, it holds that
\begin{equation*}
\lim_{i \to \infty} \dist(-\nabla f(x_i), \regnorcone{\sparse{2}{1}}{x_i}) = 0.
\end{equation*}
Thus, the measure of B-stationarity $\sparse{2}{1} \to \R : x \mapsto \dist(-\nabla f(x), \regnorcone{\sparse{2}{1}}{x})$ is not lower semicontinuous at $(0, 0)$, and the convergence to an M-stationary point that is not B-stationary cannot be suspected based on the mere observation of this limit. In the terminology of \cite{LevinKileelBoumal2023}, $((0, 0), (x_i)_{i \in \N}, f)$ is an apocalypse.

\begin{figure}[h]
\begin{center}
\begin{tikzpicture}[scale=2]
\def\a{1}
\def\b{1}
\pgfmathsetmacro\alph{0.45}
\draw [->] (-0.3, 0) -- ({\a+0.5}, 0) node [right] {$\R \times \{0\}$};
\draw [->] (0, -0.3) -- (0, {\b+0.5}) node [above] {$\{0\} \times \R$};
\draw (\a, 0) node {{\tiny $\bullet$}} node [below] {$(a, 0)$};
\draw (0, \b) node {{\tiny $\bullet$}} node [left] {$(0, b)$};
%P2GD iterations
\foreach \i in {0, 1, ..., 4}
{
%Iterates
\draw (0, {\b*(1-\alph)^\i}) node {{\tiny $\bullet$}};
%Negative gradient multiplied by the step size
\draw [->] (0, {\b*(1-\alph)^\i}) -- ({\alph*\a}, {\b*(1-\alph)^(\i+1)});
}
\end{tikzpicture}
\begin{tikzpicture}[scale=2]
\def\a{1}
\def\b{1}
\pgfmathsetmacro\alph{0.45}
\pgfmathsetmacro\istar{floor(ln(\alph*abs(\a)/abs(\b))/ln(abs(1-\alph)))}
\draw [->] (-0.3, 0) -- ({\a+0.5}, 0) node [right] {$\R \times \{0\}$};
\draw [->] (0, -0.3) -- (0, {\b+0.5}) node [above] {$\{0\} \times \R$};
\draw (\a, 0) node {{\tiny $\bullet$}} node [below] {$(a, 0)$};
\draw (0, \b) node {{\tiny $\bullet$}} node [left] {$(0, b)$};
%P2GD iterations
\foreach \i in {0, 1, ..., \istar}
{
%Iterates
\draw (0, {\b*(1-\alph)^\i}) node {{\tiny $\bullet$}};
%Negative gradient multiplied by the step size
\draw [->] (0, {\b*(1-\alph)^\i}) -- ({\alph*\a}, {\b*(1-\alph)^(\i+1)});
}
%PGD iterations
\foreach \i in {1, 2, ..., 4}
{
\draw ({\a*(1-(1-\alph)^\i)}, 0) node {{\tiny $\bullet$}};
}
\end{tikzpicture}
\end{center}
\caption{First few iterates generated by $\pgd$ (right) and $\ppgd$ (left) on the instance of~\eqref{eq:MinDiffFunctionClosedSet} studied in section~\ref{sec:PGDvsP2GD} with $a \coloneq b \coloneq 1$ and $\alpha \coloneq 0.45$. The arrows represent $x_i-\alpha \nabla f(x_i)$. The point $(a, 0)$, which is the unique global minimizer, is also represented. It is already visible from the first few iterates that $\ppgd$ converges to the M-stationary point $(0, 0)$ while $\pgd$ converges to the global minimizer.}
\label{fig:PGDvsP2GD}
\end{figure}

\section{Conclusion}
\label{sec:Conclusion}
The main contribution of this paper is the proof of Theorem~\ref{thm:StationarityAccumulationPointsPGD}. This theorem ensures that $\pgd$ (Algorithm~\ref{algo:PGD}) enjoys the strongest stationarity properties that can be expected for problem~\eqref{eq:MinDiffFunctionClosedSet} under the considered assumptions.

A sufficient condition for the convergence of a sequence generated by $\pgd$ is provided in Theorem~\ref{thm:PGDaccumulatesAtBStationaryPoints}. However, if satisfied, this condition does not offer a characterization of the rate of convergence. This important matter is addressed in \cite{JiaKanzowMehlitz} for monotone $\pgd$ under the assumption that $f$ satisfies \ref{it:Lipschitz_diff} and a Kurdyka--{\L}ojasiewicz property.

Two possible extensions of this work are left for future research.
First, can Theorem~\ref{thm:StationarityAccumulationPointsPGD} be extended to an algorithm that uses other descent directions than the negative gradient? For example, a search direction at a point $x \in C$ that is not B-stationary for~\eqref{eq:MinDiffFunctionClosedSet} could be a vector $v \notin \proxnorcone{C}{x}$ that satisfies \cite[(2) and (3)]{GrippoLamparielloLucidi}, i.e., $\ip{\nabla f(x)}{v} \le - c_1 \norm{\nabla f(x)}^2$ and $\norm{v} \le c_2 \norm{\nabla f(x)}$ with $c_1, c_2 \in (0, \infty)$.

Second, can Theorem~\ref{thm:StationarityAccumulationPointsPGD} be extended to the proximal gradient algorithm as defined in \cite[Algorithm~4.1]{KanzowMehlitz} or \cite[Algorithm~3.1]{DeMarchi}? The first step toward such an extension would be defining suitable stationarity notions for the corresponding problem whose objective function is not differentiable. Furthermore, significant adaptations would be needed, e.g., because inequality~\eqref{eq:ProjectedTranslationDistance}, which plays an instrumental role in our analysis, does not seem to admit a straightforward extension.

\appendix

\section{Proof of Proposition~\ref{prop:ProximalRegularLowSparsityOrRank}}
\label{app:examples_proof}
The strict inclusion follows from \cite[Proposition~7.16]{OlikierGallivanAbsil2023} and \cite[Theorem~3.9]{BauschkeLukePhanWang2014} if $C = \sparse{n}{s}$, from \cite[Proposition~6.7]{OlikierGallivanAbsil2023} and \cite[Theorem~3.4]{Tam2017} if $C = \sparse{n}{s} \cap \R_+^n$, from \cite[Corollary~2.3 and Theorem~3.1]{HosseiniLukeUschmajew2019} if $C = \R_{\le r}^{m \times n}$, and from \cite[Propositions~6.28 and 6.29]{OlikierGallivanAbsil2023} if $C = \mathrm{S}_{\le r}^+(n)$. By \eqref{eq:NestedNormalCones}, it remains to prove that, for all $x \in C$, $\proxnorcone{C}{x} \supseteq \regnorcone{C}{x}$. This follows from \cite[Lemma~4]{AbsilMalick} if $x \in S_p$. Let $x \in C \setminus S_p$. If $C$ is $\sparse{n}{s}$ or $\R_{\le r}^{m \times n}$, then, by \cite[Proposition~7.16]{OlikierGallivanAbsil2023} and \cite[Corollary~2.3]{HosseiniLukeUschmajew2019}, $\regnorcone{C}{x} = \{0\}$ and the result follows. If $C$ is $\sparse{n}{s} \cap \R_+^n$ or $\mathrm{S}_{\le r}^+(n)$, then the result follows from \cite[Proposition~6.7]{OlikierGallivanAbsil2023} and \cite[Proposition~3.2]{Tam2017} or \cite[Proposition~6.28]{OlikierGallivanAbsil2023} and \cite[Corollary~17]{Dax2014}; details are given below for completeness.

Assume that $C$ is $\sparse{n}{s} \cap \R_+^n$. Let $\supp(x) \coloneq \{i \in \{1, \dots, n\} \mid x_i \ne 0\}$. By \cite[Proposition~6.7]{OlikierGallivanAbsil2023},
\begin{equation*}
\regnorcone{\sparse{n}{s} \cap \R_+^n}{x} = \{v \in \R_-^n \mid \supp(v) \subseteq \{1, \dots, n\} \setminus \supp(x)\}.
\end{equation*}
Thus, by \cite[Proposition~3.2]{Tam2017}, for every $v \in \regnorcone{\sparse{n}{s} \cap \R_+^n}{x}$, $\proj{\sparse{n}{s} \cap \R_+^n}{x+v} = \{x\}$.

Assume now that $C$ is $\mathrm{S}_{\le r}^+(n)$. Let $I_n$ denote the identity matrix in $\R^{n \times n}$ and $\mathrm{O}(n) \coloneq \{U \in \R^{n \times n} \mid UU^\tp = I_n\}$ be an orthogonal group. By \cite[Proposition~6.28]{OlikierGallivanAbsil2023},
\begin{equation*}
\regnorcone{\mathrm{S}_{\le r}^+(n)}{X} = \mathrm{S}(n)^\perp + \{Z \in \mathrm{S}^-(n) \mid XZ = 0_{n \times n}\}
\end{equation*}
with $\mathrm{S}(n) \coloneq \{X \in \R^{n \times n} \mid X^\tp = X\}$, $\mathrm{S}(n)^\perp = \{X \in \R^{n \times n} \mid X^\tp = -X\}$, and $\mathrm{S}^-(n) \coloneq \{X \in \mathrm{S}(n) \mid X \preceq 0\}$. Let $Z \in \regnorcone{\mathrm{S}_{\le r}^+(n)}{X}$ and $Z_\mathrm{sym} \coloneq \frac{1}{2}(Z+Z^\tp)$. Then, by \cite[Corollary~17]{Dax2014}, $\proj{\mathrm{S}_{\le r}^+(n)}{X+Z} = \proj{\mathrm{S}_{\le r}^+(n)}{X+Z_\mathrm{sym}}$. Let $\ushort{r} \coloneq \rank X$ and $\tilde{r} \coloneq \rank Z_\mathrm{sym}$. Since $\im Z_\mathrm{sym} \subseteq \ker X$, $\tilde{r} \le n-\ushort{r}$ and there is $U \in \mathrm{O}(n)$ such that
\vspace{-1mm}
\begin{equation*}
X = U \diag(\lambda_1(X), \dots, \lambda_{\ushort{r}}(X), 0_{n-\ushort{r}}) U^\tp
\end{equation*}
\vspace{-1mm}
and
\vspace{-1mm}
\begin{equation*}
Z_\mathrm{sym} = U \diag(0_{n-\tilde{r}}, \lambda_{n-\tilde{r}+1}(Z_\mathrm{sym}), \dots, \lambda_n(Z_\mathrm{sym})) U^\tp
\end{equation*}
are eigenvalue decompositions. Hence,
\begin{equation*}
X+Z_\mathrm{sym} = U \diag(\lambda_1(X), \dots, \lambda_{\ushort{r}}(X), 0_{n-\ushort{r}-\tilde{r}}, \lambda_{n-\tilde{r}+1}(Z_\mathrm{sym}), \dots, \lambda_n(Z_\mathrm{sym})) U^\tp
\end{equation*}
is an eigenvalue decomposition. Thus, by \cite[Corollary~17]{Dax2014}, $\proj{\mathrm{S}_{\le r}^+(n)}{X+Z_\mathrm{sym}} = \{X\}$.

\section*{Acknowledgment}
The authors thank two anonymous referees for several helpful comments that improved the quality of the paper.

\bibliographystyle{siamplain}
\bibliography{PGD_bib}
\end{document}